\documentclass[journal]{IEEEtran}
\usepackage{setspace}
\usepackage{enumitem}
\usepackage[dvipsnames]{xcolor}
\usepackage{graphicx}
\usepackage{epstopdf}
\usepackage{csquotes}
\usepackage{mathtools}
\usepackage{amssymb, amsmath,amsthm, amsfonts}
\usepackage{ifthen}
\usepackage{tikz}
\usepackage{cite}
\usepackage{verbatim}
\usepackage{pifont}
\usepackage{bm}  
\usepackage{stackengine}
\usepackage{bbm}

\allowdisplaybreaks
\usepackage{accents}

\usepackage{xpatch}
\makeatletter
\xpatchcmd{\@thm}{\thm@headpunct{.}}{\thm@headpunct{}}{}{}
\makeatother

\usepackage{graphicx}
\usepackage{subcaption}
\usepackage{enumitem}
\usepackage{hyperref}
\usepackage{url}
\usepackage{wrapfig,framed}
\usepackage{setspace}
\usepackage{graphicx}
\usepackage{epstopdf}
\usepackage{csquotes}
\usepackage{mathtools}
\usepackage{amssymb, amsmath,amsthm, amsfonts}
\usepackage{ifthen}
\usepackage{tikz}
\usepackage{cite}
\usepackage[round]{natbib}
\usepackage{verbatim}
\usepackage{pifont}
\usepackage{bm}  
\usepackage{stackengine}
\usepackage{bbm}

\allowdisplaybreaks
\usepackage{accents}

\usepackage{xpatch}
\makeatletter
\xpatchcmd{\@thm}{\thm@headpunct{.}}{\thm@headpunct{}}{}{}
\makeatother

\DeclareMathOperator*{\argmax}{arg\,max}

\DeclareSymbolFont{bbold}{U}{bbold}{m}{n}
\DeclareSymbolFontAlphabet{\mathbbold}{bbold}

\newcommand{\cB}{\mathcal{B}}

\newcommand{\cF}{\mathcal{F}}
\newcommand{\cG}{\mathcal{G}}
\newcommand{\cH}{\mathcal{H}}
\newcommand{\cI}{\mathcal{I}}

\newcommand{\cN}{\mathcal{N}}

\newcommand{\cP}{\mathcal{P}}

\newcommand{\cS}{\mathcal{S}}

\newcommand{\cX}{\mathcal{X}}

\newcommand{\EE}{\mathbb{E}}

\newcommand{\NN}{\mathbb{N}}

\newcommand{\PP}{\mathbb{P}}

\newcommand{\RR}{\mathbb{R}}

\newcommand*{\dd}{\, \mathsf{d}}
\newcommand{\es}{\mathsf{ess\,sup}}
\newcommand{\vasti}{\bBigg@{3.5 }}
\newcommand{\vast}{\bBigg@{4}}
\newcommand{\Vast}{\bBigg@{5}}
\newcommand{\Vastt}{\bBigg@{7}}
\newtheorem{theorem}{Theorem}
\newtheorem{cor}{Corollary}
\newtheorem{lemma}{Lemma}
\newtheorem{prop}{Proposition}
\newtheorem{remark}{Remark}
\newtheorem{definition}{Definition}

\newcommand{\norm}[1]{\left\Vert#1\right\Vert}
\newcommand{\abs}[1]{\left\vert#1\right\vert}

\newcommand{\ind}{\mathbbm 1}

\newcommand{\Ucal}{\mathcal{U}}

\newcommand{\X}{\mathcal{X}}

\newcommand\blfootnote[1]{%
	\begingroup
	\renewcommand\thefootnote{}\footnote{#1}%
	\addtocounter{footnote}{-1}%
	\endgroup
}

\newcommand{\kl}[2]{\mathsf{D}_\mathsf{KL}\left(#1\middle\|#2\right)}
\newcommand{\chisq}[2]{\chi^2\left(#1\middle\|#2\right)}

\def\h2{\tilde h}

\def\hm1{\hat{\mathsf{H}}_{-1}}

\definecolor{ceruleanblue}{rgb}{0.16, 0.32, 0.75}
\newcommand{\revised}[1]{{\color{black} #1}}

\title{Non-asymptotic Performance Guarantees for Neural Estimation of $\mathsf{f}$-Divergences}

\begin{document}
\title{Non-Asymptotic Performance Guarantees for\\ Neural Estimation of $\mathsf{f}$-Divergences}

\author{Sreejith Sreekumar, Zhengxin Zhang,   Ziv Goldfeld }
\maketitle
\begin{abstract}
Statistical distances (SDs), which quantify the dissimilarity between probability distributions, are central to machine learning and statistics. A modern method for estimating such distances from data relies on parametrizing a variational form by a neural network (NN) and optimizing it. These estimators are abundantly used in practice, but corresponding performance guarantees are partial and call for further exploration. In particular, there seems to be a fundamental tradeoff between the two sources of error involved: approximation and estimation. While the former needs the NN class to be rich and expressive, the latter relies on controlling complexity. This paper explores this tradeoff by means of non-asymptotic error bounds, focusing on three popular choices of SDs---Kullback-Leibler divergence, chi-squared divergence, and squared Hellinger distance. Our analysis relies on non-asymptotic function approximation theorems and tools from empirical process theory. Numerical results validating the theory are also provided.
\end{abstract}
\blfootnote{The work of S. Sreekumar is supported by the TRIPODS Center for Data Science National Science Foundation Grant CCF-1740822. The work of Z. Goldfeld is supported in part by the NSF CRII Award under Grant CCF-1947801, in part by the 2020 IBM Academic Award, and in part by the NSF CAREER Award under Grant CCF-2046018.
S. Sreekumar  and Z. Goldfeld are with the School of Electrical and Computer Engineering, Cornell University, Ithaca, NY 14850, USA. Z. Zhang is with the Center for Applied Mathematics at the same University. (email: sreejithsreekumar@cornell.edu;   zz658@cornell.edu; goldfeld@cornell.edu).}
\section{INTRODUCTION}
Statistical distances (SDs) measure the discrepancy between probability distributions. A variety of machine learning (ML) tasks, from generative modeling \citep{kingma2013auto,goodfellow2014generative,nowozin2016f,arjovsky2017wasserstein,tolstikhin2018wasserstein} to those relying on barycenters \citep{rabin2011wasserstein,gramfort2015fast,dognin2019wasserstein}, can be posed as measuring or optimizing a SD between the data distribution and the model. Popular SDs include $\mathsf{f}$-divergences 
\citep{ali1966general,csiszar1967information}, integral probability metrics (IPMs) \citep{zolotarev1983probability,muller1997integral}, and Wasserstein distances \citep{villani2008optimal}. A common formulation that captures many of these is\footnote{Specifically, \eqref{Mestgen} accounts for $\mathsf{f}$-divergences, IPMs and the 1-Wasserstein distance.}
\begin{equation}
     \mathsf{H}_{\gamma, \mathcal{F}}(P,Q)=\sup_{f \in \mathcal{F}} \mathbb{E}_P[f]-\mathbb{E}_Q[\gamma \circ f], \label{Mestgen}
\end{equation}
where $\cF$ is a function class of `discriminators' and $\gamma$ is sometimes called a `measurement function' (cf., e.g., \cite{arora2017generalization}). This variational form is at the core of many ML algorithms implemented based on SDs, and has been recently leveraged for estimating SDs from samples.

\revised{Various non-parametric estimators of SDs are available in the literature \citep{Wang-2005,Perez-2008,krishnamurthy2014nonparametric,kandasamy2015nonparametric,liang2019estimating}. These classic methods typically rely on kernel density estimation (KDE) or $k$-nearest neighbors (kNN) techniques, and are known to achieve optimal estimation error rates for specific SDs, subject to smoothness and/or regularity conditions on the densities. To mention a few, \cite{kandasamy2015nonparametric} proposed a KDE-based estimator which achieves the parametric mean squared error (MSE) rate for KL divergence estimation, provided that the densities are bounded away from zero and belong to a H\"{o}lder class of sufficiently large smoothness. For the special case of entropy estimation in the high smoothness regime,  \cite{berrett2019efficient} proposed an asymptotically efficient weighted kNN estimator that does not rely on the boundedness from below assumption. Recently, \cite{Yanjun-2020} proposed a minimax rate-optimal entropy estimator for densities of sufficient Lipschitz smoothness. While these classic estimators achieve optimal performance under appropriate assumptions, they are often hard to compute in high dimensions.}
\subsection{Statistical Distances Neural Estimation}

Typical applications to machine learning, e.g., generative adversarial networks (GANs) \citep{nowozin2016f,arjovsky2017wasserstein} or anomaly detection \citep{Poczos-liang-2011}, favor estimators whose computation scales well with number of samples and is compatible with backpropagation and minibatch-based optimization. A modern estimation technique that adheres to these requirements is the so-called neural estimation method \citep{arora2017generalization,zhang2018discrimination,belghazi2018}.  
Neural estimators (NEs) parameterize the discriminator class $\mathcal{F}$ in \eqref{Mestgen} by a neural network (NN), approximate expectations by sample means, and then optimize the obtained empirical objective. Denoting the samples from $P$ and $Q$ by $X^n:=(X_1,\cdots,X_n)$ and $Y^n:=(Y_1,\cdots,Y_n)$, respectively, the resulting NE is
\begin{equation}
    \hat{\mathsf{H}}_{\gamma,\mathcal{G}_k}(X^n,Y^n):=\sup_{g \in \mathcal{G}_k} \frac 1n \sum_{i=1}^n\Big[ g(X_i)- \gamma \circ g(Y_i)\Big], \label{Mest-emp}
\end{equation}
where $\mathcal{G}_k$ is the class of functions realizable by a $k$-neuron NN. Despite the popularity of NEs in applications, their theoretical properties and corresponding performance guarantees remain largely obscure. Addressing this deficit is the objective of this work.

There is a fundamental tradeoff between the quality of approximation by NNs and the sample size needed for accurate estimation of the parametrized form. The former is measured by the \textit{approximation error}, $\big|\mathsf{H}_{\gamma, \mathcal{F}}(P,Q)-\mathsf{H}_{\gamma, \mathcal{G}_k}(P,Q)\big|$, whereas the latter by the \textit{estimation error}, $\big|\hat{\mathsf{H}}_{\gamma,\mathcal{G}_k}(X^n,Y^n)-\mathsf{H}_{\gamma, \mathcal{G}_k}(P,Q)\big|$. While approximation needs $\cG_k$ to be rich and expressive, efficient estimation relies on controlling its complexity. Past works on NEs provide only a partial account of estimation performance. \cite{belghazi2018} proved consistency of mutual information neural estimation (MINE), which boils down to estimating KL divergence, but do not quantify approximation errors. Non-asymptotic sample complexity bounds for the parameterized form, i.e., when $\cF$ in \eqref{Mestgen} is the NN class $\cG_k$ to begin with, were derived in \cite{arora2017generalization,zhang2018discrimination}. These objects are known as NN distances and, by definition, overlook the approximation error with respect to (w.r.t.) the original SD. Also related is \cite{Nguyen-2010}, where KL divergence estimation rates are provided under the assumption that the approximating class is large enough to contain an optimizer of \eqref{Mestgen}. This assumption is often violated in practice, e.g., when using NNs as done herein, or a reproducing kernel Hilbert space, as considered in \cite{Nguyen-2010}. This makes the quantification of the approximation error pivotal for a complete account of estimation performance. In light of the above, our objective is to derive non-asymptotic neural estimation performance bounds that characterize the dependence of the error on $k$ and $n$, and help understand tradeoffs between~them. 
\subsection{Contributions}
We show that the effective (approximation plus estimation) error of a NE realized by a $k$-neuron shallow NN with bounded parameters and $n$ samples scales like $$O\left(k^{-1/2}+h_{\gamma}(k) n^{-1/2}\right),$$ where $h_{\gamma}(k)$ grows with $k$ at a rate that depends on the estimated SD. In order to bound the approximation error, we refine Theorem 1 in \citet{Barron-1992} to show that a $k$-neuron NN with \emph{bounded parameters} can approximate any function in the Barron class \citep{Barron_1993} under the sup-norm within an $O(k^{-1/2})$ error.  To control the empirical estimation error, we leverage tools from empirical process theory and bound the associated entropy integral \citep{AVDV-book} to achieve the $O\left(h_{\gamma}(k)n^{-1/2}\right)$ convergence rate.

The effective error bound is then specialized to three predominant $\mathsf{f}$-divergences: KL, chi-squared ($\chi^2$ divergence, and squared Hellinger distance. We establish finite-sample absolute-error bounds of these NEs by identifying the appropriate scaling of the width $k$ with the sample size $n$ in the general bounds. This, in turn, implies consistency of the NEs. Our analysis is based on two key observations. First, to achieve a small approximation error, we would like $\cG_k$ to universally approximate the original function class $\cF$, which needs either width \citep{lu2017expressive} or parameters \citep{stinchcombe1990approximating} to be unbounded. On the other hand, to achieve the parametric estimation rate $n^{-1/2}$, the class $\cG_k$ must not be too large. The effective error bound then relies on finding the appropriate scaling of $k$ (and the uniform parameter norm) with $n$ so that a small approximation error and fast estimation rates are both attained. Numerical results (on synthetic data) validating our theory are also provided.
\subsection{Notation}   Let $\| \cdot \|$ denote the Euclidean norm on $\RR^d$, and $x \cdot y$ designate the inner product.~The Euclidean ball of radius $r\geq 0$ centered at 0 is $B^d(r)$. 
We use $\bar{\RR}:=\RR\cup \{-\infty,\infty\}$ for the extended reals. For $1\leq p<\infty$, the $L^p$ space over $\cX\subseteq\RR^d$ w.r.t. the Lebesgue measure is denoted by $L^p(\cX)$, with $\| \cdot \|_p$ designating the norm. 
We let $(\Omega,\mathcal{A},\PP)$ be the probability space on which all random variables are defined;~$\EE$ denotes the corresponding expectation. The class of Borel probability measures on $\X \subseteq \RR^d$ is denoted by $\cP(\X)$. To stress that an expectation of  $f$ is taken w.r.t. $P\in\cP(\X)$, we write $\EE_P[f]$. We assume that all functions considered henceforth are Borel functions. The essential supremum of a function w.r.t. $P\in\cP(\X)$ is denoted by $\es_P(f)$. For $P,Q\in\cP(\X)$ with $P \ll Q$, i.e., $P$ is absolutely continuous w.r.t. $Q$, we use $\frac{\dd P}{\dd Q}$ for the Radon-Nikodym derivative of $P$ w.r.t. $Q$. For $n \in \mathbb{N}$, $P^{\otimes n}$ denotes the $n$-fold product measure of $P$. For an open set $\Ucal \subseteq \RR^d$ and an integer $m \geq 0$, the class of functions such that all partial derivatives of order $m$ exist and are continuous on $\Ucal$ are denoted by $\mathsf{C}^m(\Ucal)$. In particular, $\mathsf{C}(\Ucal):=\mathsf{C}^0(\Ucal)$ and $\mathsf{C}^{\infty}(\Ucal)$ denotes the class of continuous functions and infinitely differentiable functions on $\Ucal$. The restriction of a function  $f:\RR^d\to\RR$ to a subset $\cX \subseteq \RR^d$ is represented by $f|_\cX$.  For $a,b \in \RR$, $a \vee b:=\max\{a,b\}$ and $a \wedge b:=\min\{a,b\}$. For a multi-index $\boldsymbol{\alpha}=(\alpha_1, \cdots,\alpha_d)$, $D^{\boldsymbol{\alpha}}:=\frac{\partial^{\alpha_1}}{\partial^{\alpha_1} x_1}\cdots \frac{\partial^{\alpha_d}}{\partial^{\alpha_d} x_d}$ denotes the partial derivative operator of order $\abs{\boldsymbol{\alpha}}:=\sum_{j=1}^d \alpha_j$.
\section{Background and Preliminaries} \label{sec:prbform}
  Below, we provide a short background on the central technical ideas used in the paper. 
\paragraph{Statistical distances.} A common variational formulation of a SDs between $P,Q\in\cP(\X)$, $\X \subseteq \RR^d$,~is
\begin{align}
     \mathsf{H}_{\gamma, \mathcal{F}}(P,Q)=\sup_{f \in \mathcal{F}} \mathbb{E}_P[f]-\mathbb{E}_Q[\gamma\circ f], \label{Mestgen2}
\end{align}
where $\gamma:\mathbb{R}\rightarrow \bar{\mathbb{R}}$, and $\mathcal{F}$ is a class of measurable functions $f:\RR^d\rightarrow \mathbb{R}$ for which the expectations are finite. This formulation captures  $\mathsf{f}$-divergences (when $\gamma$ is the convex conjugate of $\mathsf{f})$, IPMs (for $\gamma(x)=x$) as well as the 1-Wasserstein distance (which is an IPM w.r.t. the 1-Lipschitz function class). 
\begin{figure*}[!t]
\setcounter{equation}{10}
\begin{equation} \label{EQ:NN_class}
  \mathcal{G}_k(\mathbf{a}) :=\mspace{-3mu}\left\{ g:\mathbb{R}^d\mspace{-5mu}\rightarrow \mathbb{R}:\begin{aligned}
    &g(x)=\sum_{i=1}^k \beta_i \phi\left(w_i\cdot x+b_i\right) +b_0,~ w_i \in \mathbb{R}^d,~ b_0,b_i,\beta_i \in \mathbb{R},\\ &\max\nolimits_{\substack{i=1,\ldots,k\\j=1,\ldots,d}}\left\{|w_{i,j}|,|b_i|\right\} \leq a_1,~
  |\beta_i| \leq a_2,~ ~i=1,\ldots,k,~|b_0| \leq a_3  \end{aligned}\right\}\mspace{-3mu}.
\end{equation}
\setcounter{equation}{3}
\hrulefill 
\end{figure*}
\paragraph{Approximated function class.} Our approximation result requires the target function with domain $\X$ to have an extension on $\RR^d$, which  belongs to a certain class of functions introduced  in \cite{Barron_1993}.
    \begin{definition}[Barron class] \label{def:barronclass}
Consider a function $f:\mathbb{R}^d \rightarrow \mathbb{R}$ that has a Fourier representation  $ f(x)=\int_{0}^{\infty} e^{i\omega \cdot x} \tilde F(d \omega)$, where $\tilde F(d\omega)$ is a complex Borel measure over $\mathbb{R}^d$ with magnitude $F(d\omega)$ that satisfies
\begin{equation} \label{Cfconstdefval}
   B(f,\X):= \int_{\RR^d} \sup_{x \in \X} \abs{\omega \cdot x} F(d \omega) <\infty.
\end{equation}
 For $c \geq 0$, the Barron class is
\begin{flalign}
   \cB_c(\X)\mspace{-4 mu} :=\mspace{-4 mu} \left\{ f:\mathbb{R}^d \rightarrow \mathbb{R}, B(f,\X) \vee \abs{f(0)} \mspace{-3 mu} \leq \mspace{-3 mu} c\right\}. \label{fourintclas}
\end{flalign}
\end{definition}
For $\tilde f:\X \rightarrow \RR$, define 
\begin{align}
    c_B^\star(\tilde f,\X):=\mspace{-4 mu} \inf\left\{c:\exists~  f \in  \cB_c(\X),~\tilde f= f|_\cX \right\}. \label{barroncoeff}
\end{align}
\paragraph{Stochastic processes.} Our analysis of the estimation error requires the following definitions.

\begin{definition}[Subgaussian process] Let $(\Theta,d)$ be a metric space. A real-valued stochastic process $\{X_{\theta}\}_{\theta \in \Theta}$ with index set $\Theta$ is called subgaussian if it is centered and 
\begin{equation}
    \mathbb{E}\big[e^{t(X_{\theta}-X_{\theta'})}\big] \leq e^{\frac 12 t^2d(\theta,\theta')^2},~ \forall ~\theta,\theta' \in \Theta,~t \geq 0.
\end{equation}
\end{definition}
\begin{definition}[Separable process]
A stochastic process $\{X_{\theta}\}_{\theta \in \Theta}$ on a metric space $(\Theta,d)$ is called separable if there exists a countable set $\Theta_0 \subseteq \Theta $, such that
\begin{align}
 \lim_{\theta' \rightarrow \theta:\ \theta' \in \Theta_0} X_{\theta'} \rightarrow X_{\theta}, ~\forall~ \theta \in \Theta\qquad\mbox{a.s.}
\end{align}
\end{definition}    

\begin{definition}[Covering  number] \label{cov-pack-num}
A set $\Theta'$ is an $\epsilon$-covering for the metric space $(\Theta,d)$  if for every $\theta \in \Theta$, there exists a $\theta' \in \Theta' $ such that $d(\theta,\theta')\leq \epsilon$. 
The $\epsilon$-covering  number is 
\begin{equation}
  N(\Theta,d,\epsilon):=\inf \left\{|\Theta'|:~\Theta' \mbox{ is an } \epsilon \mbox{-covering for }\Theta \right\}.  \notag
\end{equation}
\end{definition}
The next theorem gives a tail bound for the supremum of a subgaussian process in terms of the covering~number. This result is key for our estimation error analysis. 
\begin{theorem}{\cite[Theorem 5.29]{VanHandel-book}}\label{thm:tailineq}
Let $\{X_{\theta}\}_{\theta \in \Theta }$ be a separable subgaussian process on the metric space $(\Theta,d)$. Then, for any $\theta_0 \in \Theta$ and $\delta \geq 0$, we have
\begin{align}
  &  \mathbb{P}\left(\sup_{\theta \in \Theta} X_{\theta}-X_{\theta_0} \geq C \int_{0}^{\infty} \sqrt{\log N(\Theta,d,\epsilon)}d\epsilon+\delta \right) \notag \\
  &\leq Ce^{-\frac{\delta^2}{C\mathsf{diam}(\Theta)^2}}, \label{probtailineqent}
\end{align}
for $\mathsf{diam}(\mspace{-1.5mu}\Theta\mspace{-1.5mu})\mspace{-3mu}:=\mspace{-8mu} \sup\limits_{\theta,\theta' \in \Theta}\mspace{-6mu}d(\mspace{-1mu}\theta\mspace{-3mu},\mspace{-2mu}\theta'\mspace{-1mu})$ and a universal~constant~$C$. 
\end{theorem}

\begin{figure*}[!t]
\setcounter{equation}{13}
\begin{equation} \label{bndedholderclass}
 \cH_{b,c}^{l,\delta}(\Ucal):=\left\{f \in \mathsf{C}^l(\Ucal):\ \ \begin{aligned} &
 \abs{D^{\boldsymbol{\alpha}}f(x)-D^{\boldsymbol{\alpha}}f(x')} \leq c\norm{x-x'}^{\delta},~\forall x,x' \in \Ucal,  ~\abs{\boldsymbol{\alpha}}=l,\\& \max_{\boldsymbol{\alpha}:\abs{\boldsymbol{\alpha}}  \leq l} \sup_{x \in \Ucal} \abs{D^{\boldsymbol{\alpha}}f(x)} \leq b \end{aligned}\right\}. 
\end{equation}\setcounter{equation}{9}
\hrulefill 
\end{figure*}

\section{Statistical Distances Neural Estimation}
For simplicity of presentation, we henceforth fix $\X=[0,1]^d$, although our results and analysis readily generalize to arbitrary compact supports  $\X \subset \RR^d$. Accordingly,  $B(f,\X)$, $\cB_c(\X)$ and $ c_B^\star(\tilde f,\X)$ are denoted by $B(f)$, $\cB_c$ and $ c_B^\star(\tilde f)$, respectively. We first describe the neural estimation method, followed by two technical results that account for the approximation and the estimation errors. These results are later leveraged to derive effective error bounds for neural estimation of KL and $\chi^2$ divergences, as well as the squared Hellinger distance. All proofs are deferred to the~supplement.
\subsection{Neural Estimation}
Let $P,Q\in\cP(\X)$. Consider a SD $\mathsf{H}_{\gamma,\cF}(P,Q)$ between these distributions (see \eqref{Mestgen2}), and assume that $n$ independently and identically distributed (i.i.d.) samples $X^n:=(X_1,\cdots,X_n)$ and $Y^n:=(Y_1,\cdots,Y_n)$ from $P$ and $Q$, respectively, are available. The NE of $\mathsf{H}_{\gamma,\cF}(P,Q)$ based on a $k$-neuron shallow network (to parametrize the function class $\mathcal{F}$) and the samples $X^n,Y^n$ (to approximate the expected values) is 
\begin{equation}
    \hat{\mathsf{H}}_{\gamma,\mathcal{G}_k(\mathbf{a})}(X^n,Y^n) \mspace{-3 mu}:= \mspace{-5 mu}\sup_{g \in \mathcal{G}_k(\mathbf{a})} \frac 1n \sum_{i=1}^n \mspace{-3 mu}\Big[ g(X_i)- \gamma \circ g(Y_i)\Big], \label{Mest-emp2}
\end{equation}
where $\mathcal{G}_k(\mathbf{a})$ is the NN class defined~in \eqref{EQ:NN_class} above, with parameter bounds specified by  $\mathbf{a}=(a_1,a_2,a_3)$ $\in\RR^3_{\geq 0}$, and activation function $\phi:\mathbb{R}\rightarrow\mathbb{R}$, which is henceforth taken as the logistic sigmoid $\phi(x)=\frac{1}{1+e^{-x}}$. The results that follow extend to any measurable bounded variation sigmoidal (i.e., $\phi(z) \rightarrow 1$ as $z \rightarrow \infty$ and $\phi(z) \rightarrow 0$ as $z \rightarrow -\infty$) activation.

Our goal is to provide absolute-error performance guarantees for this NE, in terms of the approximation error and the statistical estimation error.\footnote{In practice, an optimization error is also present, but its exploration is left for future work.}


\subsection{Sup-norm Function Approximation}
We start with a bound on the approximation error of a target function $\tilde f$ with domain $\X$ for which $ c_B^\star(\tilde f)<\infty$.
\begin{theorem}[Approximation]\label{THM:approximation} \label{supaapprox}~Let~$P,Q\in\cP(\X)$ and consider the NN class  $\cG_k^*\left(c\right):=\cG_k\left( \sqrt{k}\log k, 2k^{-1}c,c\right)$ (see \eqref{EQ:NN_class}) for some $c \geq 0$. 
\setcounter{equation}{11}
Given  $\tilde f:\X \rightarrow \mathbb{R}$ such that $ c_B^\star(\tilde f)\leq c$, there exists $g \in \mathcal{G}_k^*\left(c\right)$ satisfying
\vspace{-1mm}
\begin{align}
\big\|\tilde f-g\big\|_{\infty, P,Q}= O\left( k^{-\frac 12}\right), \label{approxratefin}
\vspace{-3mm}
\end{align}
where  $\norm{f\mspace{-3mu}-\mspace{-3mu}g}_{\infty,P,Q}\mspace{-3mu}:=\es_{P}|f\mspace{-1.5mu}-\mspace{-1.5mu}g| \vee \es_{Q}|f\mspace{-1.5mu}-\mspace{-1.5mu}g|$.   Moreover, for any $\mathbf{a} \in \RR^3_{\geq 0}$, $c>0$,  and $\epsilon>0$, there exists $\tilde f:\X \rightarrow \mathbb{R}$ with $ c_B^\star(\tilde f)\leq c$ such that 
\begin{align}
\inf_{g \in \cG_k(\mathbf{a})}\big\|\tilde f-g\big\|_{2}= \Omega\left( k^{-\big(\frac 12+\frac 1d+\epsilon\big)}\right). \label{approxratelb}
\vspace{-3mm}
\end{align}
The explicit dependence of $d$ and $c$ in the right hand side (R.H.S.) of \eqref{approxratefin} is given in \eqref{finaapperrbarcls}.
\end{theorem}
The above theorem states that a $k$-neuron shallow NN can approximate a function $\tilde f$ on  $\X$ within an $O(k^{-1/2})$ gap in the uniform norm, provided  $\tilde f$ is the restriction of some $f$ from the Barron class. The upper and lower bounds in \eqref{approxratefin} and \eqref{approxratelb} differ by $k^{-(1/d+\epsilon)}$, which becomes negligible for large $d$ and small $\epsilon$. Also observe that the lower bound is in terms of $L^2$ norm which implies a lower bound w.r.t. the $L^{\infty}$ norm. 

\begin{remark}[Relation to previous results]
A result reminiscent to Theorem \ref{THM:approximation} appears in \cite{Barron-1992}, but some technical details had to be adapted to apply the bound to neural estimation of SDs. Theorem \ref{supaapprox} 
generalizes \citet[Theorem 2]{Barron-1992} and \citet[Theorem 2.2]{Yukich-1995} from unbounded NN weights and bias parameters to bounded ones. We note that while \citet[Theorem 2.2]{Yukich-1995} allows unbounded  parameters (input weight and bias), a more general problem of~approximating a function and its derivatives is treated therein. Here we only consider the approximation of the function itself.
\end{remark}
We next show that a sufficiently smooth H\"{o}lder function on $\X$ is the restriction of some function in the Barron class. To that end we first define the H\"{o}lder function class.
\begin{definition}[Holder class]
  For $b,c,\delta \geq 0$, an integer $l \geq 0$, and an open set $\Ucal \subseteq \RR^d$, the bounded holder class $\cH_{b,c}^{l,\delta}(\Ucal)$ is defined in \eqref{bndedholderclass} above.
\end{definition}
We have the following universal approximation property for H\"{o}lder functions.
\begin{cor}[Approximation of H\"{o}lder functions] \label{cor:bndfourcoeff}
Given  a function  $\tilde f:\X \rightarrow \RR$, suppose there exists an open set $~\Ucal \supset  \X$,  $b,c,\delta \geq 0$, and $ f \in  \cH_{b,c}^{s,\delta}(\Ucal)$, $s:=\lfloor\frac{d}{2}\rfloor+2$,  such that $\tilde f= f|_\cX$.  
Then, there exists $g \in \mathcal{G}_k^*\left(\bar c_{b,c,d}\right)$ such~that
\setcounter{equation}{14}
\begin{align}
\vspace{-4mm}
\big\|\tilde f-g\big\|_{\infty, P,Q}= O\left(k^{-\frac 12}\right),
\vspace{-4mm}
\end{align}
where $\bar c_{b,c,d}$ is given in \eqref{constapproxhold} in Appendix \ref{cor:bndfourcoeff-proof}.
\end{cor}
\subsection{Estimation of Parameterized Distances}

We next bound the error of estimating the parametrized SD $\mathsf{H}_{\gamma, \mathcal{G}_k(\mathbf{a})}(P,Q)$ (i.e., \eqref{Mestgen2} for a NN function class) by its empirical version from \eqref{Mest-emp2}. Throughout this section we assume that $X^n$ and $Y^n$ are, respectively, i.i.d. samples from $P$ and $Q$. 
\begin{theorem}[Empirical estimation error tail bound] \label{empesterrbnd}
Let $P,Q\in\cP(\X)$. 
Assume that $\mathbf{a}\in\RR^3_{\geq 0}$, $\mathcal{G}_k(\mathbf{a})$, and $\gamma:\RR\to\bar{\RR}$ are such that $\mathsf{H}_{\gamma,\mathcal{G}_k(\mathbf{a})}(P,Q)<\infty$ \revised{and \begin{align}
    \bar{\gamma}'_{\mathcal{G}_{k}(\mathbf{a})}:=\sup_{\substack{x \in \X, \\g \in \mathcal{G}_k(\mathbf{a})} }\gamma' \circ g(x)<\infty, \label{maxdergamma}
\end{align}}%
where $\gamma'$ is the derivative of $\gamma$. Then, for  the universal constant $C$ from Theorem \ref{thm:tailineq} and any $\delta>0$, we have
\begin{flalign}
   & \mathbb{P}\mspace{-1mu}\Big(\mspace{-2mu}\abs{ \hat{\mathsf{H}}_{\gamma,\mathcal{G}_k(\mathbf{a})}\mspace{-1.5mu}(X^n\mspace{-3mu},\mspace{-3mu}Y^n\mspace{-1.5mu})\mspace{-2mu}-\mspace{-2mu}\mathsf{H}_{\gamma, \mathcal{G}_k(\mathbf{a})}\mspace{-2mu}(P\mspace{-3mu},Q)}\mspace{-2mu}\geq\mspace{-2mu} \delta\mspace{-2mu}+\mspace{-2mu}C E_{k,\mathbf{a},n,\gamma}\Big) \notag \\
    &\qquad\qquad\qquad\qquad\qquad\quad\quad\ \ \leq  2C e^{-\frac{n\delta^2}{V_{k,\mathbf{a},\gamma}}}, \label{bndesterremp} &&
\end{flalign}
where \revised{$E_{k,\mathbf{a},n,\gamma}=O\left(n^{-1/2}\right)$} and explicit expressions for $V_{k,\mathbf{a},\gamma}$ and $E_{k,\mathbf{a},n,\gamma}$ are given in \eqref{Vkconstdef}-\eqref{Ekconstdef} in Appendix~\ref{empesterrbnd-proof}. 
\end{theorem}
The proof of Theorem \ref{empesterrbnd} (see Appendix \ref{empesterrbnd-proof}) involves upper bounding the estimation error by a separable subgaussian process and invoking Theorem \ref{thm:tailineq}.

\begin{remark}[NN distances]
The SD $\mathsf{H}_{\gamma, \mathcal{G}_k(\mathbf{a})}(P,Q)$ is the so-called NN distance, studied in \cite{arora2017generalization,zhang2018discrimination} in the context of GANs. Theorem \ref{empesterrbnd} can be understood as a sample complexity bound for NN distance estimation from data. Taking $\delta$ of order $n^{-1/2}$, the estimation error attains the parametric rate with high probability.
\end{remark}

\subsection{$\mathsf{f}$-Divergence Neural Estimation}\label{SUBSEC:KL_NE}
Having Theorems \ref{supaapprox}-\ref{empesterrbnd} and Corollary \ref{cor:bndfourcoeff}, we analyze neural estimation of three important SDs: KL divergence, $\chi^2$ divergence and squared  Hellinger distance.
\setcounter{equation}{22}
\begin{figure*}[!t]
\vspace{-3mm}
\begin{equation}
  \mspace{-10 mu}  \mathcal{L}_{\mathsf{KL}}(b,c)\mspace{-3 mu}:=\mspace{-3 mu}\left\{\mspace{-3 mu}(P,Q)\mspace{-3 mu} \in\mspace{-3 mu} \mathcal{P}_{\mathsf{KL}}(\X)  \mspace{-2 mu} : \begin{aligned}    & \exists ~  f,\bar f \in \cH_{b,c}^{s,\delta}(\Ucal)\mbox{ for some $\delta\geq 0$ and open set }\Ucal\supset \mspace{-3 mu} \X ~\mbox{ s.t.}\mspace{1 mu}\log p=f|_\X,\\&\log q=\bar f|_\X\end{aligned} \right\}. \label{classklerrbnd}
\end{equation}
\hrulefill 
\setcounter{equation}{17}
\end{figure*}
\subsubsection*{KL Divergence} \label{sec:KLdiv}
The KL divergence between $P,Q\in\cP(\X)$ with $P\ll Q$ is $\kl{P}{Q}:=\mathbb{E}_P\left[\log \left(\frac{\dd P}{\dd Q}\right)\right]$ (and infinite when $P$ is not absolutely continuous w.r.t. $Q$). A variational form for $\kl{P}{Q}$ is obtained via Legendre-Fenchel duality, yielding: 
\begin{equation}
\kl{P}{Q}=\sup_{f :\X \rightarrow \mathbb{R}}\mathbb{E}_P[f]-\mathbb{E}_Q\left[e^{f}-1\right], \label{CC-charact}
\end{equation}
where the supremum is over all measurable functions such that expectations are finite. This fits the framework of \eqref{Mestgen2} with $\gamma(x)=\gamma_{\mathsf{KL}}(x):=e^x-1$. The supremum in \eqref{CC-charact} is achieved by $f_{\mathsf{KL}}:=\log\left(\frac{\dd P}{\dd Q}\right)$.

Let $\hat{D}_{\mathcal{G}_k(\mathbf{a}_k)}(X^n,Y^n):=\hat{\mathsf{H}}_{\gamma_{\mathsf{KL}},\mathcal{G}_k(\mathbf{a}_k)}(X^n,Y^n)$ be a NE of $\kl{P}{Q}$, where $\mathbf{a}_k\in\RR^3_{\geq 0}$ for all $k\in\NN$. The effective error achieved by the estimator can be bounded as the sum of the approximation and estimation errors. 

To present error bounds, we require a few definitions.  
 Let $\mathcal{P}_{\mathsf{KL}}(\X)$ be the set of all pairs $(P,Q) \in \mathcal{P}(\X) \times \mathcal{P}(\X)$ such that $P \ll Q$ and $\kl{P}{Q}<\infty$, and set  
\begin{equation}
 \cI(m):= \left\{  f:\X \rightarrow \RR,~ c_B^\star(f)\vee \norm{f}_{\infty} \leq m  \right\}.\label{EQ:I_set}
\end{equation}
As a consequence of  proof of Corollary  \ref{cor:bndfourcoeff}, $\cI(m)$ is non-empty since  it contains any $f\in\cH_{b,c}^{l,\delta}(\Ucal)$ for some $ \Ucal \supseteq \X $ and appropriately chosen parameters $l,b,c,\delta$. For any $m$, the aforementioned condition is satisfied, e.g., by Gaussian densities with suitable parameters.

The following theorem establishes the consistency of  $\hat{D}_{\mathcal{G}_k(\mathbf{a}_k)}(X^n,Y^n)$ and bounds the effective (approximation and estimation) error in terms of the NN and sample sizes, which reveals the tradeoff between them.
\begin{theorem}[KL neural estimation]\label{strongcons}
Let  
$(P,Q) \in \mathcal{P}_{\mathsf{KL}}(\X)$. For any $\alpha>0:$
\begin{enumerate}[label = (\roman*),leftmargin=15 pt]
 \item If  $f_{\mathsf{KL}} \in \mathsf{C}\left(\X\right)$, then 
 for  $\{k_n\}_{n \in \NN}$, $n$ such that $k_n\ \rightarrow  \infty$ and $k_n \leq (\frac 12-\alpha) \log n$, 
 \begin{equation}
 \mspace{-20 mu}
   \hat{D}_{\mathcal{G}_{k_n}(\mathbf{1})}(X^n,Y^n)    \xrightarrow[n\rightarrow \infty]{} \kl{P}{Q},\quad \mathbb{P}-\mbox{a.s.} \label{finbndascon}
 \end{equation}
 \item 
Suppose there exists an $M$ such that $f_{\mathsf{KL}} \in \mathcal{I}(M)$. Then, for $k$ and $n$   such that $k^3=O\left(n^{1-\alpha}\right)$,
\begin{flalign}
 &  \mathbb{E}\left[  \abs{\hat{D}_{\mathcal{G}_{k}^*(0.5 \log k)}(X^n,Y^n)  -\kl{P}{Q}}\right] \notag \\
 &  \qquad \qquad \qquad   = O\left(k^{-\frac{1}{2}}+ k^{\frac 32} n^{-\frac 12}\right).\label{KLeffbndsimp} &&
\end{flalign} 
\end{enumerate}
\end{theorem}
The consistency result (Part $(i)$) in the above theorem uses 
the fact that $\mathcal{G}_{k_n}(\mathbf{1})$ is a universal approximator for the class of continuous functions on compact sets as $k_n \rightarrow \infty$. The error bound in \eqref{KLeffbndsimp}  utilizes Theorems \ref{supaapprox}-\ref{empesterrbnd} to bound the effective error as the sum of the approximation and estimation errors. From \eqref{approxratefin}, the former error is $O(k^{-1/2})$ if $c_B^\star\left(f_{\mathsf{KL}}\right) \leq M$ and  $\mathbf{a}$ is such that $\cG_k^*\left(M\right) \subseteq \mathcal{G}_k(\mathbf{a})$. 
As $M$ is often unavailable (due to $P$ and $Q$ being unknown), in order to achieve the above error, we  take  $\mathcal{G}_k(\mathbf{a}_k)=\cG_k^*(m_k)$ for some increasing positive  sequence $\{m_k\}_{k \in \NN}$ ($m_k=0.5 \log k$ in Theorem \ref{strongcons} above) such that $m_k \rightarrow \infty$. This  ensures that $m_k\geq M$ for sufficiently large $k$. 
\begin{remark}[KL approximation-estimation tradeoff] \label{KLNEratessimp}
In Appendix \ref{strongcons-proof}, we state the KL neural estimation error bound for  an arbitrary increasing sequence $\{m_k\}_{k\in\NN}$ (see \eqref{finerrbndklapp}).  
If $M$ (such that $f_{\mathsf{KL}} \in \mathcal{I}(M)$) is known when picking the NN parameters, then for a network of size $k =O\left(n^{(1-\alpha)}\right)$ with $m_k=M$, we have (see Remark \ref{remarkeffratekl} in Appendix~\ref{strongcons-proof})
 \begin{flalign}
 &  \mathbb{E}\left[  \abs{\hat{D}_{\mathcal{G}_{k}^*(M)}(X^n,Y^n)  -\kl{P}{Q}}\right] \notag \\
 &\qquad\qquad \qquad \qquad \quad~ = O\left(k^{-\frac{1}{2}}+\sqrt{k}~ n^{-\frac 12}\right).\label{finbnderrrate-case1} &&
 \end{flalign}
\end{remark}

\begin{remark}[KL effective sample complexity]  \label{rem:KLsampcomp}
The optimal choice of $k$ for \eqref{finbnderrrate-case1} is $k=\sqrt{n}$ (for $\alpha<0.5$). Inserting this into \eqref{finbnderrrate-case1}, we obtain the effective error bound $O\left(n^{-1/4}\right)$.
Although this rate is polynomial in~$n$, it is slower than the parametric $n^{-1/2}$ rate that can be achieved for KL divergence estimation via KDE techniques in the very smooth density regime \citep{kandasamy2015nonparametric}.\footnote{The latter relies on a different technical assumption in terms of H\"{o}lder-smoothness of underlying densities.} 
\end{remark}

\begin{remark}[$L^2$ neural estimation of a function] A reminiscent analysis for the sample complexity of learning a NN approximation of a bounded range function from samples was employed in \cite{Barron-1994}. 
This differs from our setup since SDs are given as a supremum over a function class as opposed to a single function. As such, our results require stronger sup-norm approximation results, as opposed to the $L^2$ bound used in \cite{Barron-1994}. 
\end{remark}
\setcounter{equation}{15}
Theorem  \ref{strongcons} provides conditions on $f_{\mathsf{KL}}$ under which bounds on the effective error of neural estimation can be obtained (namely, that $f_{\mathsf{KL}} \in \mathcal{I}(M)$ for some $M$). 
A primitive condition in terms of the densities of $P$ and $Q$ is given next. 
Let $\mu$ be a measure that dominates both $P$ and $Q$, i.e., $P,Q \ll \mu$, and denote the corresponding densities by $p:=\frac{\dd P}{\dd \mu}$ and $q:=\frac{\dd Q}{\dd \mu}$.

   \begin{figure*}[!t]
\vspace{-3mm}
\setcounter{equation}{27}
\begin{equation}
 \mspace{-10 mu}\mathcal{L}_{\chi^2}(b,c)\mspace{-3 mu}:=\mspace{-3 mu}\left\{\mspace{-3 mu}(P,Q) \mspace{-3 mu}\in\mspace{-3 mu} \mathcal{P}_{\chi^2}(\X): \begin{aligned}&\exists ~f,\bar f \in \cH_{b,c}^{s,\delta}(\Ucal)\mbox{ for some $\delta \geq 0$ and open set }\Ucal \mspace{-3 mu}\supset \X \mspace{-3 mu} \mbox{ s.t. }p\mspace{-3 mu}=\mspace{-3 mu}f|_\X, \\& q^{-1}\mspace{-3 mu}=\mspace{-3 mu}\bar f|_\X\end{aligned}\right\}. \label{classchisqerrbnd}
\end{equation}
    \hrulefill
    \vspace{-3mm}
\end{figure*}
\begin{prop}[KL sufficient condition] \label{prop:distconddirect}
For  $b,c \geq 0$, consider the class $\mathcal{L}_{\mathsf{KL}}(b,c)$ of pairs of distributions defined in \eqref{classklerrbnd} above. Suppose $(P,Q) \in  \mathcal{L}_{\mathsf{KL}}(b,c)$. Then, Part (ii) of Theorem \ref{strongcons} and \eqref{finbnderrrate-case1}  hold with  $M=2\bar c_{b,c,d}$, with  $\bar c_{b,c,d}$ as defined in Corollary~\ref{cor:bndfourcoeff}.

\end{prop}

\begin{remark}\label{kldivclassdist}[Feasible distributions]
For  appropriately chosen $b,c \mspace{-4 mu}\geq \mspace{-4 mu} 0$, the class $\mathcal{L}_{\mathsf{KL}}(b,c)$  contains distribution pairs $(P,Q) \mspace{-4 mu}\in \mspace{-4 mu}\mathcal{P}_{\mathsf{KL}}(\X)$ whose densities w.r.t. a common dominating measure (e.g., $(P\mspace{-3 mu}+\mspace{-3 mu}Q)/2$) are bounded (from above and below) on $\X$ with a smooth extension on an open set covering $\X$. In particular, this includes uniform distributions, truncated Gaussians, truncated Cauchy distributions, etc. 
\end{remark}
\subsubsection*{$\chi^2$ Divergence}
The $\chi^2$ (chi-squared) divergence between $P,Q\mspace{-3mu}\in\mspace{-3mu}\mathcal{P}(\X)$ ~with~$P\mspace{-3mu}\ll\mspace{-3mu}Q$~is
$\chisq{P}{Q}\mspace{-2mu}=\mspace{-2mu}\mathbb{E}_{Q}\left[\left(\mspace{-2mu}\frac{\dd P}{\dd Q}\mspace{-2mu}-\mspace{-2mu}1\mspace{-1mu}\right)^2\mspace{-1mu}\right]$ (and infinite when $P$ is not absolutely continuous w.r.t. $Q$). It admits the dual~form:
\setcounter{equation}{23}
\begin{equation}
     \chisq{P}{Q}=\sup_{\substack{f:\X\rightarrow \mathbb{R}}} \mathbb{E}_P[f]-\mathbb{E}_Q\left[f+f^2/4\right], \label{chisqdistvarchar}
\end{equation}
where the supremum is over all $f$ such that expectations are finite. This dual form corresponds to \eqref{Mestgen2} with $\gamma(x)=\gamma_{\chi^2}(x):=x+\frac{x^2}{4}$. The supremum in \eqref{chisqdistvarchar} is achieved by $f_{\chi^2}=2\left(\frac{\dd P}{\dd Q}-1\right)$.
 
Let $\hat{\chi}^2_{\mathcal{G}_k(\mathbf{a}_k)}(X^n,Y^n):=\hat{\mathsf{H}}_{\gamma_{\chi^2},\mathcal{G}_k(\mathbf{a}_k)}(X^n,Y^n)$ denote the NE of $\chisq{P}{Q}$. Set $\mathcal{P}_{\chi^2}(\X)$ as the collection of all $(P,Q) \in \mathcal{P}(\X) \times \mathcal{P}(\X)$ such that $P \ll Q$ and $\chisq{P}{Q}<\infty$. The next theorem establishes consistency of the NE and bounds its effective absolute-error.
\begin{theorem}[$\chi^2$ neural estimation]\label{strongconschisq}
Let  
$(P,Q) \in \mathcal{P}_{\chi^2}(\X)$. For any $\alpha>0:$
\begin{enumerate}[label = (\roman*),leftmargin=15 pt]
     \item If $f_{\chi^2} \in \mathsf{C}\left(\X\right)$,
then for $\{k_n\}_{n \in \NN}$, $n$ such that $k_n \rightarrow  \infty$ and  $k_n =O\left(n^{(1-\alpha)/5}\right)$, 
 \begin{align}
 \mspace{-20 mu}  \hat{\chi}^2_{\mathcal{G}_{k_n}(\mathbf{1})}(X^n,Y^n)    \xrightarrow[n\rightarrow \infty]{}  \chisq{P}{Q}, \quad \mathbb{P}-\mbox{a.s.}  \label{finbndasconchisq}
 \end{align}
\item Suppose there exists an $M$ such that $f_{\chi^2}\in \mathcal{I}(M)$ (see \eqref{EQ:I_set}). Then, for $k$ and  $n$ such that $\sqrt{k} \log^2 k=O\left(n^{(1-\alpha)/2}\right)$, we have
  \begin{flalign}
 &  \mathbb{E}\left[  \abs{\hat{\chi}^2_{\mathcal{G}_{k}^*(0.5 \log k)}(X^n,Y^n) -\chisq{P}{Q}}\right] \notag \\
 & \qquad\qquad \qquad = O\left(k^{-\frac{1}{2}}+ \sqrt{k} \log^2k~ n^{-\frac 12}\right). \label{chisqsimprate} &&
 \end{flalign}
   \end{enumerate}
  \end{theorem}
The proof of Theorem \ref{strongconschisq} (see Appendix \ref{strongconschisq-proof})  is similar to that of Theorem \ref{strongcons}. 
\begin{remark}[$\chi^2$ effective sample complexity]
In Appendix \ref{strongconschisq-proof}, we obtain general error bounds (see \eqref{finerrbndchisqapp}) assuming an arbitrary increasing sequence $\{m_k\}_{k \in \NN}$, as mentioned in Remark \ref{KLNEratessimp}.  
Given $M$ with $f_{\chi^2} \in \mathcal{I}(M)$, for $m_k=M$ and $k= O\left(n^{(1-\alpha)}\right)$, we have
 \begin{flalign}
     & \mathbb{E}\left[  \abs{\hat{\chi}^2_{\mathcal{G}_{k}^*(M)}(X^n,Y^n) -\chisq{P}{Q}}\right]  = O\mspace{-4mu}\left(\mspace{-2mu}k^{-\frac 12}\mspace{-4mu}+\mspace{-4mu}\sqrt{k}~ n^{-\frac 12}\mspace{-3mu}\right). \label{finbnderrratechisq-case1} &&
 \end{flalign}
 Comparing \eqref{finbndasconchisq}-\eqref{finbnderrratechisq-case1} to \eqref{finbndascon}-\eqref{finbnderrrate-case1}, we see that consistency holds under milder conditions and that the effective error bound is slightly better for $\chi^2$ divergence than for KL divergence. As in Remark \ref{rem:KLsampcomp}, the optimal choice of $k$ in \eqref{finbnderrratechisq-case1}  is $k=\sqrt{n}$ (for $\alpha<0.5$). This results in an effective error bound of $O(n^{-1/4})$. 
\end{remark}

The next result is the counterpart of Proposition \ref{prop:distconddirect} to $\chi^2$ divergence (see Appendix \ref{prop:distconddirect-chisq-proof} for proof). 
\begin{prop}[$\chi^2$ sufficient condition] \label{prop:distconddirect-chisq}
 For  $b,c \geq 0$, consider the class $\mathcal{L}_{\chi^2}(b,c)$ of pairs of distributions defined in \eqref{classchisqerrbnd} above, and suppose that $(P,Q) \in  \mathcal{L}_{\chi^2}(b,c)$. 
 Then, Part (ii) of Theorem \ref{strongconschisq} and \eqref{finbnderrratechisq-case1} hold with  $M=(2+\bar  c^2_{b,c,d}2^{\lfloor d/2\rfloor+3})(\kappa_d \sqrt{d}\vee 1) $, where $\kappa_d$ and $\bar c_{b,c,d}$ are given in \eqref{constkappa} and \eqref{constapproxhold}, respectively.
\end{prop}
\begin{remark}[Feasible distributions]The class $\mathcal{L}_{\chi^2}(b,c)$, for appropriately chosen $b,c \geq 0$, contains all $(P,Q) \in \mathcal{P}_{\chi^2}(\X)$,  whose densities $p,q,$ w.r.t. a common dominating measure are bounded (upper bounded for $p$ and bounded away from zero for $q$) on $\X$ with an extension that is sufficiently smooth on an open set covering $\X$. This includes the distributions mentioned in Remark~\ref{kldivclassdist}.
\end{remark}

\setcounter{equation}{28}
\subsubsection*{Squared Hellinger distance}
The squared Hellinger distance between~$P\mspace{-3mu},\mspace{-1.5mu}Q\mspace{-3mu}\in\mspace{-2mu}\mathcal{P}(\X\mspace{-1mu})$  with $P \ll Q$ is $ H^2(P,Q):=\mathbb{E}_Q\bigg[\Big(\sqrt{\frac{\dd P}{\dd Q}}-1\Big)^2\bigg]$, and 
\begin{equation}
 H^2(P,Q)=\mspace{-7mu}\sup_{\substack{f:\X \rightarrow \mathbb{R},\\ f(x)< 1,\forall x \in \X}}\mspace{-7mu} \mathbb{E}_P[f]\mspace{-2mu}-\mspace{-2mu}\mathbb{E}_Q\big[f/(1-f)\big], \label{helvarchar}
\end{equation}
is its dual form, where the supremum is over all functions such that the expectations are finite (\eqref{helvarchar} corresponds to~\eqref{Mestgen2} with $\gamma(x)=\gamma_{H^2}(x):=\frac{x}{1-x}\Big)$. The supremum in~\eqref{helvarchar} is achieved by $f_{H^2}=1-\left(\frac{\dd P}{\dd Q}\right)^{-\frac 12}$. 
\begin{figure*}[!t]
\hspace{-0.5mm}
\centering
 \begin{subfigure}[t]{0.165\textwidth}
 \centering
  \includegraphics[width=1.2\linewidth]{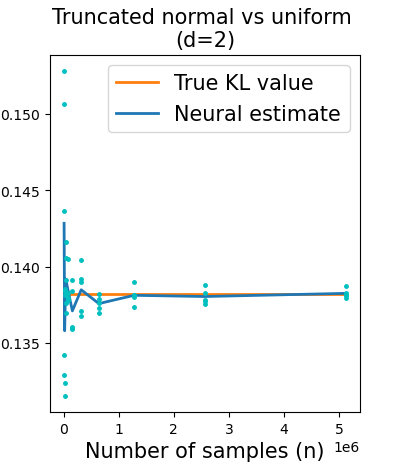}
   \vspace{-4.5mm}
   \caption{} \label{KL-2d-estvsnumsamp}
   \vspace{-1mm}
 \end{subfigure}
 \ \hspace{0.1mm}
 \begin{subfigure}[t]{0.165\textwidth}
 \centering
 \includegraphics[width=1.2\linewidth]{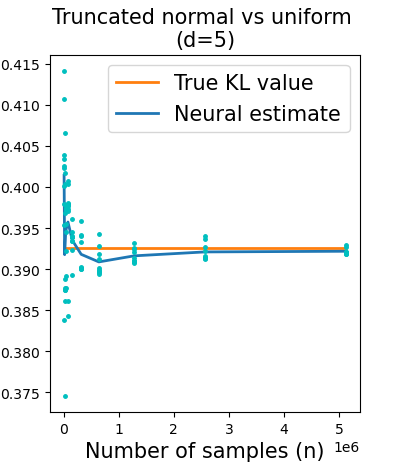}
    \vspace{-4.5mm}
   \caption{}\label{fig:5DGaussUniform_a}
   \vspace{-1mm}
 \end{subfigure}
 \ \hspace{-0.5mm}
 \begin{subfigure}[t]{0.165\textwidth}
 \centering
  \includegraphics[width=1.2\linewidth]{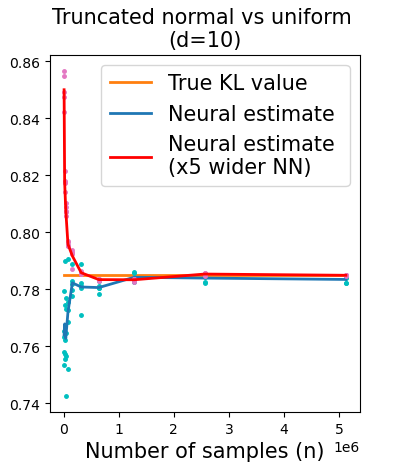}
   \vspace{-4.5mm}
   \caption{} \label{KL-10d-estvsnumsamp}
   \vspace{-1mm}
 \end{subfigure}
  \ \hspace{-1mm}
 \begin{subfigure}[t]{0.165\textwidth}
 \centering
  \includegraphics[width=1.2\linewidth]{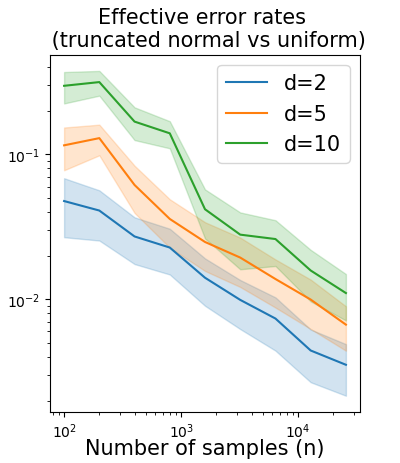}
   \vspace{-4.5mm}
   \caption{} \label{error}
   \vspace{-1mm}
 \end{subfigure}
 \ \hspace{-0.8mm}
 \begin{subfigure}[t]{0.265\textwidth}
 \centering
 \includegraphics[width=1.1\linewidth]{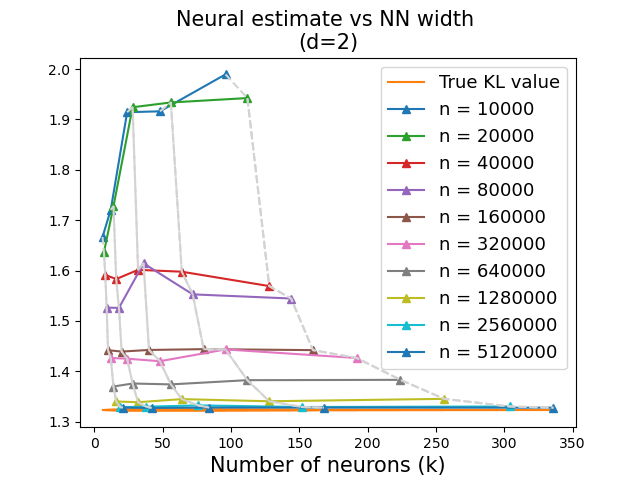}
 \vspace{-4.5mm}
 \caption{}\label{KL-2d-estvsk}
 \vspace{-1mm}
 \end{subfigure}
  \caption{Neural estimate of KL divergence: (a) estimate versus $n$ convergence in dimension $d=2$, for  $P$ given by $\cN(\mathbf{0},\mathrm{I}_2)$ 
  truncated to be supported inside $\X=[0.1,2]\times[-1,0]$ and $Q=\mathsf{Unif}(\X)$; (b) estimate versus $n$ convergence in dimension $d=5$, for $P$  given by $\cN(\mathbf{0},\mathrm{I}_5)$ truncated to $\X=[0.1,2]\times[-1,0]\times[2,3]\times[-2,-1.5]\times[-1,1]$ and $Q=\mathsf{Unif}(\X)$; (c) similar to (b) but in dimension $d=10$ and with compact support $\X\times\X$; (d) effective error rates versus number of samples; (e) estimate versus number of neurons $k$, for fixed number of samples $n$.}
  \label{fig:2DGauss}
\vspace{-2mm}
\end{figure*}
Let $
\hat{H}^2_{\tilde {\mathcal{G}}_k(\mathbf{a}_k,t)}(X^n,Y^n):=\hat{\mathsf{H}}_{\gamma_{ H^2},\tilde{\mathcal{G}}_k(\mathbf{a}_k,t)}(X^n,Y^n)$,
where $t>0$ and $\tilde{\mathcal{G}}_{k}(\mathbf{a},t)$ is the NN class
  \begin{equation}
  \mspace{-10 mu} \tilde{\mathcal{G}}_{k}\mspace{-3 mu}\left(\mathbf{a},t\right)
  \mspace{-3 mu}:=\mspace{-3 mu}\left\{\mspace{-3 mu}g:\mathbb{R}^d \mspace{-3 mu}\rightarrow \mathbb{R}: \begin{aligned}& g(x)\mspace{-3mu}=\mspace{-3mu}(1-t) \wedge \tilde g(x), \\&\tilde g \in \mathcal{G}_k(\mathbf{a})\end{aligned}\right\}.  \label{truncNNf}
  \end{equation}
Set
\begin{equation}
\mathcal{I}_{H^2}(m):= \mspace{-3 mu}\left\{  f:\X \rightarrow \RR, \begin{aligned}& c_B^\star(f) \vee  \mspace{-3 mu}\norm{(1-f)^{-1}}_{\infty}\mspace{-5 mu} \\&\vee \norm{f}_{\infty} \mspace{-5 mu}\leq m  \end{aligned}\right\}, \notag
\end{equation}
and $\mathcal{P}_{H^2}(\X)$ as the collection of all $(P,Q) \in \mathcal{P}(\X) \times \mathcal{P}(\X)$ such that $P \ll Q$ (note that $0 \leq H^2(P,Q) \leq 2$). Define the shorthands $\tilde{\cG}_{k,t}^{(1)}=\tilde{\mathcal{G}}_{k}(\mathbf{1},t)$, \begin{equation} \tilde{\cG}_{k,m,t}^{(2)}:=\tilde{\cG}_k\left( \sqrt{k}\log k, 2k^{-1}m,m,t\right). \notag
\end{equation} The next theorem establishes consistency of the NE and bounds its effective absolute-error (see Appendix \ref{strongconshel-proof} for proof).
\begin{theorem}[$H^2$ neural estimation]\label{strongconshel}
Let  
$(P,Q) \in \mathcal{P}_{H^2}(\X)$. For any $\alpha>0:$
\setcounter{equation}{30}
\begin{enumerate}[label = (\roman*),leftmargin=15 pt]
     \item If $f_{H^2} \in \mathsf{C}\left(\X\right)$, 
then, for $\{k_n,t_{k_n}\}_{n \in \NN}$, such that $ k_n \rightarrow \infty$,  $t_{k_n}>0, ~t_{k_n} \rightarrow 0$, and $k_n^{\frac 32}t_{k_n}^{-2} =O\left(n^{(1-\alpha)/2}\right)$,
\begin{align}
  \mspace{-10mu} \hat{H}^2_{\tilde{\mathcal{G}}_{k_n,t_{k_n}}^{(1)}}\mspace{-4mu}(X^n,Y^n)    \xrightarrow[n\rightarrow \infty]{}  H^2(P,Q), \quad \mathbb{P}-\mbox{a.s.}  \label{finbndasconhel}
\end{align}
\item Suppose there exists $M$ such that $f_{H^2} \in \mathcal{I}_{H^2}(M)$. Then, for $k,n$
with
$\log^3 k \sqrt{k}=O\left(n^{(1-\alpha)/2}\right)$, $m_k=0.5 \log k$ and $t_k=\log^{-1} k$, we have
  \begin{flalign}
 &  \mathbb{E}\left[  \abs{\hat{H}^2_{\tilde{\mathcal{G}}_{k,m_k,t_k}^{(2)}}(X^n,Y^n) -H^2(P,Q)}\right] \notag \\
 &\qquad \quad= \left(\log k~ k^{-\frac{1}{2}}\mspace{-2.5mu}\right)+O\left(\log^3 k \sqrt{k}~n^{\mspace{-2mu}-\frac 12}\mspace{-2mu}\right). \label{helsimperrbnd}&&
 \end{flalign} 
   \end{enumerate}
  \end{theorem}
To establish effective error bounds for squared Hellinger distance, we used a truncated NN class $\tilde{\mathcal{G}}_{k}\mspace{-3 mu}\left(\mathbf{a},t\right)$ given in \eqref{truncNNf}, which is the function class obtained by saturating the shallow NN output to $1-t$ for some $t>0$. This is done since $\gamma_{H^2}(x)$ has a singularity at $x=1$ and the NN outputs must be truncated below 1 so as to satisfy \eqref{maxdergamma} for bounding the empirical estimation error. For obtaining  effective error bounds under this constraint, we scale the parameter $t$ with $k$ as $\{t_k\}_{k \in \NN}$ for some decreasing positive sequence $t_k \rightarrow 0$. The bound in \eqref{helsimperrbnd} uses $t_k=\log^{-1} k$.
  \begin{remark}[Effective sample complexity] 
  In Appendix \ref{strongconshel-proof}, we obtain effective error bounds (see \eqref{finbndratehelcase2}) for an arbitrary decreasing positive sequence $\{t_k\}_{k \in \NN}$, with $t_k \rightarrow 0$, and an increasing positive divergent sequence $\{m_k\}_{k \in \NN}$.
  If $f_{H^2} \in \mathcal{I}_{H^2}(M)$ and the NN parameters can depend on $M$, then, for  $k$, $t_k=\log^{-1} k$ and $n$ such that  $\sqrt{k}\log^2 k= O\left(n^{(1-\alpha)/2}\right)$, setting $m_k=M$ in \eqref{finbndratehelcase2} yields
 \begin{flalign}
     & \mathbb{E}\left[  \abs{\hat{H}^2_{\tilde{\mathcal{G}}_{k,M,t_k}^{(2)}}(X^n,Y^n) -H^2(P,Q)}\right] \notag \\
     & \qquad \quad\  = O\mspace{-4mu}\left(\mspace{-2mu} k^{-\frac{1}{2}}\log k\mspace{-2.5mu}\right)+\mspace{-3mu}O\mspace{-3mu}\left(\sqrt{k}~\log^2 k n^{-\frac 12}\mspace{-3mu}\right). \label{finbnderrratechel-case1} &&
 \end{flalign}
The optimal choice of $k$ in \eqref{finbnderrratechel-case1} is $k=n^{\frac{1}{2(1+\eta)}}$,  $\big($for $\alpha <0.5(1-2\eta)(1+\eta)^{-1}\big)$,
where $\eta>0$ is an arbitrarily small. The resulting effective error bound is  
$O(n^{-1/4})$.
\end{remark}
\section{Empirical Results}
We illustrate the performance of KL divergence neural estimation via some simple simulations. 
The considered NN class is $\mathcal{G}_{k}^*(M)$ (see Section \ref{SUBSEC:KL_NE}) with  $M$ appropriately chosen. 
The number of samples $n$ varies from $n=10^5$ to $n=6.4 \times 10^6$, and we scale the NN size as $k=n^{1/5}$ (in accordance with $k=O\left(n^{1-\alpha}\right)$ for $\alpha>0$ sufficiently small, see \eqref{finbnderrrate-case1}). The NN is trained using Adam optimizer \citep{kingma2017adam} for 200 epochs. The initial learning rate of $10^{-2}$ is reduced to $10^{-3}$ after the first 100 epochs. We use batch size $n\times 10^{-3}$, and present plots averaged over 10 different runs (shown as dots in Figures \ref{KL-2d-estvsnumsamp}-\ref{KL-10d-estvsnumsamp}).

Figure \ref{KL-2d-estvsnumsamp} shows convergence of the NE of $\kl{P}{Q}$ versus number of samples, when $P$ is a  2-dimensional truncated Gaussians (adhering to the compact support assumption) and $Q$ is uniform distribution on the same support. For Figure \ref{KL-2d-estvsnumsamp}, we start from $\tilde{P}=\cN(\mathbf{0},\mathrm{I}_2)$, where $\mathrm{I}_d\in\RR^{d\times d}$~is the identity, and truncate (and normalize) it to $\X=[0.1,2]\times[-1,0]$ to obtain $P$, and set $Q=\mathsf{Unif}(\X)$. 
Figure \ref{fig:5DGaussUniform_a} repeats the experiment but with $P$ as a 5-dimensional Gaussian $\cN(\mathbf{0},\mathrm{I}_5)$ truncated to $\X:=[0.1,2]\times[-1,0]\times[2,3]\times[-2,-1.5]\times[-1,1]$ and $Q=\mathsf{Unif}(\X)$.  
The same setup but in dimension $d=10$ and with $\X\times\X$ (instead of $\X$) is presented in Figure \ref{KL-10d-estvsnumsamp} (blue curve). Corresponding error rates (for effective error averaged over 100 runs)  versus number of samples (on a log-log scale) are shown in Figure \ref{error}. It can be seen therein that the convergence rate is parametric for large enough values of  $n$.

While convergence is evident in all dimensions, the trajectories are different: convergence happens from above when $d$ is small and from below for large $d$ (with $d=5$ sitting in between and presenting a mixed trend). This happens because the same NN size $k=n^{1/5}$ were used in all three experiments, without factoring in the dimension (generally, higher-dimensional distribution need a larger NN). This results in the NN being relatively large when $d=2$, which causes overfitting and, in turn, overestimation of the KL divergence for small $n$ values. For $d=10$, that same NN is relatively small, resulting in underestimation for small $n$. In accordance with the above, the $d=5$ case exhibits a mixed trend. To verify this effect, we increased the NN size by a factor of 5 in the $d=10$ experiment---the obtained neural estimator is shown by the red curve in Figure \ref{KL-10d-estvsnumsamp}. As expected, the larger networks results in convergence from above, similarly to the original $d=2$ example.

To further examine the overfitting effect, Figure \ref{KL-2d-estvsk} revisits the setup of Figure \ref{KL-2d-estvsnumsamp} and shows the evolution of the neural estimate as the number of neurons $k$ grows, keeping  $n$ fixed. The neural estimate progressively gets closer to the true value as $n$ becomes larger. The overestimation of the KL divergence again highlights overfitting. The increase in the KL divergence estimate with $k$ could also be due to higher optimization error for larger $k$ values. The light gray lines across curves for different $n$ values show convergence to true $\kl{P}{Q}$. 
\section{Concluding Remarks}
This paper studied neural estimation of SDs, aiming to characterize tradeoffs between approximation and empirical estimation errors. We showed that~NEs of $\mathsf{f}$-divergences, such as the KL and $\chi^2$ divergences and the squared Hellinger distance, are consistent, provided the appropriate scaling of the NN size $k$ with the sample size $n$. We then derived non-asymptotic absolute-error upper bounds that quantify the desired tradeoff between $k$ and $n$. 
\revised{The key technical results leading to these bounds are  Theorems \ref{supaapprox}-\ref{empesterrbnd}, which, respectively, bound the sup-norm approximation error by NNs and the empirical estimation error of the parametrized~SD. }

Going forward, we aim to extend our results to additional SDs such as the total variation distance, the $1$-Wasserstein distance, etc. While the high level analysis extends to these examples, new approximation bounds for the appropriate function classes (bounded or 1-Lipschitz) are needed. Another extension of interest is to $P$ and $Q$ that are not compactly supported. This is possible within our framework under proper tail decay, but we leave the details for future work. While we have neglected the optimization error from our current analysis, this is an important component of the overall estimation error and we plan to examine it in the future. 
Lastly, generalizing our analysis to NEs based on deep neural networks is another important extension. 
Through the results herein and the said future directions, we hope to provide useful performance guarantees for NEs that would facilitate a principled usage thereof in ML applications and beyond.

\bibliography{SS_ref,ZG_ref_new}

\begin{thebibliography}{34}
\providecommand{\natexlab}[1]{#1}
\providecommand{\url}[1]{\texttt{#1}}
\expandafter\ifx\csname urlstyle\endcsname\relax
  \providecommand{\doi}[1]{doi: #1}\else
  \providecommand{\doi}{doi: \begingroup \urlstyle{rm}\Url}\fi

\bibitem[Ali and Silvey(1966)]{ali1966general}
S.~M. Ali and S.~D. Silvey.
\newblock A general class of coefficients of divergence of one distribution
  from another.
\newblock \emph{Journal of the Royal Statistical Society: Series B
  (Methodological)}, 28\penalty0 (1):\penalty0 131--142, Jan. 1966.

\bibitem[Arjovsky et~al.(2017)Arjovsky, Chintala, and
  Bottou]{arjovsky2017wasserstein}
M.~Arjovsky, S.~Chintala, and L.~Bottou.
\newblock Wasserstein generative adversarial networks.
\newblock In \emph{Proceedings of the International Conference on Machine
  Learning (ICML-2017)}, pages 214--223, Sydney, Australia, Jul. 2017.

\bibitem[Arora et~al.(2017)Arora, Ge, Liang, Ma, and
  Zhang]{arora2017generalization}
S.~Arora, R.~Ge, Y.~Liang, T.~Ma, and Y.~Zhang.
\newblock Generalization and equilibrium in generative adversarial nets
  {(GANs)}.
\newblock In \emph{Proceedings of the International Conference on Machine
  Learning (ICML-2017)}, pages 224--232, Sydney, Australia, Jul. 2017.

\bibitem[Barron(1992)]{Barron-1992}
A.~R. Barron.
\newblock Neural net approximation.
\newblock Proceedings of Seventh Yale Workshop on Adaptive and Learning
  Systems, CT, USA, 20--22 May 1992.

\bibitem[{Barron}(1993)]{Barron_1993}
A.~R. {Barron}.
\newblock Universal approximation bounds for superpositions of a sigmoidal
  function.
\newblock \emph{IEEE Transactions on Information Theory}, 39\penalty0
  (3):\penalty0 930--945, 1993.

\bibitem[Barron(1994)]{Barron-1994}
A.~R. Barron.
\newblock Approximation and estimation bounds for artificial neural networks.
\newblock \emph{Mach. Learn.}, 14\penalty0 (1):\penalty0 115–133, Jan. 1994.

\bibitem[Belghazi et~al.(2018)Belghazi, Baratin, Rajeshwar, Ozair, Bengio,
  Courville, and Hjelm]{belghazi2018}
M.~I. Belghazi, A.~Baratin, S.~Rajeshwar, S.~Ozair, Y.~Bengio, A.~Courville,
  and D.~Hjelm.
\newblock Mutual information neural estimation.
\newblock In \emph{Proceedings of the 35th International Conference on Machine
  Learning}, volume~80, pages 531--540, Stockholm Sweden, 10--15 Jul 2018.

\bibitem[Berrett et~al.(2019)Berrett, Samworth, and Yuan]{berrett2019efficient}
T.~B. Berrett, R.~J. Samworth, and M.~Yuan.
\newblock Efficient multivariate entropy estimation via $k$-nearest neighbour
  distances.
\newblock \emph{The Annals of Statistics}, 47\penalty0 (1):\penalty0 288--318,
  2019.

\bibitem[Csisz{\'a}r(1967)]{csiszar1967information}
I.~Csisz{\'a}r.
\newblock Information-type measures of difference of probability distributions
  and indirect observation.
\newblock \emph{Studia Ccientiarum Mathematicarum Hungarica}, 2:\penalty0
  229--318, 1967.

\bibitem[Dognin et~al.(2019)Dognin, Melnyk, Mroueh, Ross, Santos, and
  Sercu]{dognin2019wasserstein}
P.~Dognin, I.~Melnyk, Y.~Mroueh, J.~Ross, C.~D. Santos, and T.~Sercu.
\newblock Wasserstein barycenter model ensembling.
\newblock In \emph{Proceedings of the International Conference on Learning
  Representations (ICLR-2019)}, New Orleans, Louisiana, US, May 2019.

\bibitem[Goodfellow et~al.(2014)Goodfellow, Pouget-Abadie, Mirza, Xu,
  Warde-Farley, Ozair, Courville, and Bengio]{goodfellow2014generative}
I.~Goodfellow, J.~Pouget-Abadie, M.~Mirza, B.~Xu, D.~Warde-Farley, S.~Ozair,
  A.~Courville, and Y.~Bengio.
\newblock Generative adversarial nets.
\newblock In \emph{Proceedings of the Annual Conference on Advances in Neural
  Information Processing Systems (NeurIPS-2014)}, pages 2672--2680, 2014.

\bibitem[Gramfort et~al.(2015)Gramfort, Peyr{\'e}, and
  Cuturi]{gramfort2015fast}
A.~Gramfort, G.~Peyr{\'e}, and M.~Cuturi.
\newblock Fast optimal transport averaging of neuroimaging data.
\newblock In \emph{Proceedings of the International Conference on Information
  Processing in Medical Imaging}, pages 261--272, Hong Kong, China, Jun. 2015.

\bibitem[Han et~al.(2020)Han, Jiao, Weissman, and Wu]{Yanjun-2020}
Y.~Han, J.~Jiao, T.~Weissman, and Y.~Wu.
\newblock {Optimal rates of entropy estimation over Lipschitz balls}.
\newblock \emph{The Annals of Statistics}, 48\penalty0 (6):\penalty0 3228 --
  3250, 2020.

\bibitem[Kandasamy et~al.(2015)Kandasamy, Krishnamurthy, Poczos, Wasserman, and
  Robins]{kandasamy2015nonparametric}
K.~Kandasamy, A.~Krishnamurthy, B.~Poczos, L.~Wasserman, and J.~M. Robins.
\newblock Nonparametric {von Mises} estimators for entropies, divergences and
  mutual informations.
\newblock In \emph{Proceedings of the Annual Conference on Advances in Neural
  Information Processing Systems (NeurIPS-2015)}, pages 397--405, Montr{\'e}al,
  Canada, 2015.

\bibitem[Kingma and Ba(2017)]{kingma2017adam}
D.~P. Kingma and J.~Ba.
\newblock Adam: A method for stochastic optimization.
\newblock \emph{arXiv preprint arXiv:1412.6980}, 2017.

\bibitem[Kingma and Welling(2014)]{kingma2013auto}
D.~P. Kingma and M.~Welling.
\newblock Auto-encoding variational bayes.
\newblock In \emph{Proceedings of the International Conference on Learning
  Representations (ICLR-2014)}, Banff, Canada, Apr. 2014.

\bibitem[Krishnamurthy et~al.(2014)Krishnamurthy, Kandasamy, P{\'o}czos, and
  Wasserman]{krishnamurthy2014nonparametric}
A.~Krishnamurthy, K.~Kandasamy, B.~P{\'o}czos, and L.~Wasserman.
\newblock Nonparametric estimation of {R}{\'e}nyi divergence and friends.
\newblock In \emph{Proceedings of the International Conference on Machine
  Learning (ICML-2014)}, pages 919--927, Beijing, China, Jun. 2014.

\bibitem[Liang(2019)]{liang2019estimating}
T.~Liang.
\newblock Estimating certain {I}ntegral {P}robability {M}etric ({IPM}) is as
  hard as estimating under the {IPM}.
\newblock \emph{arXiv preprint arXiv:1911.00730}, Nov. 2019.

\bibitem[Lu et~al.(2017)Lu, Pu, Wang, Hu, and Wang]{lu2017expressive}
Z.~Lu, H.~Pu, F.~Wang, Z.~Hu, and L.~Wang.
\newblock The expressive power of neural networks: A view from the width.
\newblock In \emph{Proceedings of the Annual Conference on Advances in Neural
  Information Processing Systems (NeurIPS-2017)}, pages 6231--6239, Long Beach,
  CA, US, Dec. 2017.

\bibitem[M{\"u}ller(1997)]{muller1997integral}
A.~M{\"u}ller.
\newblock Integral probability metrics and their generating classes of
  functions.
\newblock \emph{Advances in Applied Probability}, 29\penalty0 (2):\penalty0
  429--443, 1997.

\bibitem[{Nguyen} et~al.(2010){Nguyen}, {Wainwright}, and
  {Jordan}]{Nguyen-2010}
X.~{Nguyen}, M.~J. {Wainwright}, and M.~I. {Jordan}.
\newblock Estimating divergence functionals and the likelihood ratio by convex
  risk minimization.
\newblock \emph{IEEE Transactions on Information Theory}, 56\penalty0
  (11):\penalty0 5847--5861, 2010.

\bibitem[Nowozin et~al.(2016)Nowozin, Cseke, and Tomioka]{nowozin2016f}
S.~Nowozin, B.~Cseke, and R.~Tomioka.
\newblock {$f$-GAN: T}raining generative neural samplers using variational
  divergence minimization.
\newblock In \emph{Proceedings of the Annual Conference on Advances in Neural
  Information Processing Systems (NeurIPS-2016)}, pages 271--279, Barcelona,
  Spain, Dec. 2016.

\bibitem[{Perez-Cruz}(2008)]{Perez-2008}
F.~{Perez-Cruz}.
\newblock Kullback-leibler divergence estimation of continuous distributions.
\newblock In \emph{2008 IEEE International Symposium on Information Theory},
  pages 1666--1670, 2008.

\bibitem[P\'{o}czos et~al.(2011)P\'{o}czos, Xiong, and
  Schneider]{Poczos-liang-2011}
B.~P\'{o}czos, L.~Xiong, and J.~Schneider.
\newblock Nonparametric divergence estimation with applications to machine
  learning on distributions.
\newblock In \emph{Proceedings of the Twenty-Seventh Conference on Uncertainty
  in Artificial Intelligence}, page 599–608. AUAI Press, 2011.

\bibitem[Rabin et~al.(2011)Rabin, Peyr{\'e}, Delon, and
  Bernot]{rabin2011wasserstein}
J.~Rabin, G.~Peyr{\'e}, J.~Delon, and M.~Bernot.
\newblock Wasserstein barycenter and its application to texture mixing.
\newblock In \emph{Proceedings of the International Conference on Scale Space
  and Variational Methods in Computer Vision (SSVM-2011)}, pages 435--446,
  Gedi, Israel, May 2011.

\bibitem[Stinchcombe and White(1990)]{stinchcombe1990approximating}
M.~Stinchcombe and H.~White.
\newblock Approximating and learning unknown mappings using multilayer
  feedforward networks with bounded weights.
\newblock In \emph{Proceedings of the International Joint Conference on Neural
  Networks (IJCNN-1990)}, pages 7--16, San Diego, CA, US, Jun. 1990.

\bibitem[Tolstikhin et~al.(2018)Tolstikhin, Bousquet, Gelly, and
  Sch{\"o}lkopf]{tolstikhin2018wasserstein}
I.~Tolstikhin, O.~Bousquet, S.~Gelly, and B.~Sch{\"o}lkopf.
\newblock Wasserstein auto-encoders.
\newblock In \emph{International Conference on Learning Representations
  (ICLR-2018)}, Vancouver, Canada, Apr.-May 2018.

\bibitem[{Van Der Vaart} and Wellner(1996)]{AVDV-book}
A.~{Van Der Vaart} and J.~A. Wellner.
\newblock \emph{Weak Convergence and Empirical Processes}.
\newblock Springer, New York, 1996.

\bibitem[{van Handel}(2016)]{VanHandel-book}
R.~{van Handel}.
\newblock \emph{Probability in High Dimension: Lecture Notes-Princeton
  University}.
\newblock [Online]. Available:
  \url{https://web.math.princeton.edu/~rvan/APC550.pdf}, 2016.

\bibitem[Villani(2008)]{villani2008optimal}
C.~Villani.
\newblock \emph{Optimal {T}ransport: Old and {N}ew}, volume 338.
\newblock Springer Science \& Business Media, 2008.

\bibitem[{Wang} et~al.(2005){Wang}, {Kulkarni}, and {Verdu}]{Wang-2005}
Q.~{Wang}, S.~R. {Kulkarni}, and S.~{Verdu}.
\newblock Divergence estimation of continuous distributions based on
  data-dependent partitions.
\newblock \emph{IEEE Transactions on Information Theory}, 51\penalty0
  (9):\penalty0 3064--3074, 2005.

\bibitem[{Yukich} et~al.(1995){Yukich}, {Stinchcombe}, and
  {White}]{Yukich-1995}
J.~E. {Yukich}, M.~B. {Stinchcombe}, and H.~{White}.
\newblock Sup-norm approximation bounds for networks through probabilistic
  methods.
\newblock \emph{IEEE Transactions on Information Theory}, 41\penalty0
  (4):\penalty0 1021--1027, 1995.

\bibitem[Zhang et~al.(2018)Zhang, Liu, Zhou, Xu, and
  He]{zhang2018discrimination}
P.~Zhang, Q.~Liu, D.~Zhou, T.~Xu, and X.~He.
\newblock On the discrimination-generalization tradeoff in {GANs}.
\newblock In \emph{Proceedings of the International Conference on Learning
  Representations (ICLR-2018)}, Vancouver, Canada, Apr.-May 2018.

\bibitem[Zolotarev(1983)]{zolotarev1983probability}
V.~M. Zolotarev.
\newblock Probability metrics.
\newblock \emph{Teoriya Veroyatnostei i ee Primeneniya}, 28\penalty0
  (2):\penalty0 264--287, 1983.

\end{thebibliography}


\onecolumn
\appendix

\section{Appendix}
To emphasize the underlying parameters of the NN, by some abuse of notation, we introduce 
\begin{subequations}
\begin{equation}
  \mathcal{G}_k(\Theta) :=\mspace{-3mu}\left\{ g:\mathbb{R}^d\mspace{-5mu}\rightarrow \mathbb{R}:\begin{aligned}
    &g(x)=\sum_{i=1}^k \beta_i \phi\left(w_i\cdot x+b_i\right) +b_0,~\mspace{-3mu} \big(\{\beta_i,w_i,b_i\}_{i=1}^k,b_0\big)\in\Theta \end{aligned}\mspace{-10mu}\right\},
\end{equation}
\begin{equation}
  \Theta_{k}(\mathbf{a}):=\left\{\left(\{\beta_i,w_i,b_i\}_{i=1}^k,b_0\right):\ \ \begin{aligned}&  w_i \in \mathbb{R}^d,~ b_0,b_i,\beta_i \in \mathbb{R},~\max\nolimits_{\substack{i=1,\ldots,k\\j=1,\ldots,d}}\left\{|w_{i,j}|,|b_i|\right\} \leq a_1\\&
  |\beta_i| \leq a_2,~ ~i=1,\ldots,k,~|b_0| \leq a_3  \end{aligned} 
 \right\}. \label{paramspace}
\end{equation}%
\end{subequations}
Also, throughout the Appendix, we denote $g(x)=\sum_{i=1}^k \beta_i \phi\left(w_i\mspace{-3mu}\cdot\mspace{-3mu} x\mspace{-2mu}+\mspace{-2mu}b_i\right)\mspace{-2mu}+\mspace{-2mu}b_0$ for $\theta=\big(\{\beta_i,w_i,b_i\}_{i=1}^k,b_0\big)$ by $g_{\theta}$, whenever the underlying $\theta$ needs to be emphasized.

We first state an auxiliary result which will be useful in the proofs that follow. For $b \geq 0$, an integer $l \geq 0$, and an open set $\Ucal \subseteq \RR^d$ containing the origin, consider the  class $\cS_{l,b}(\Ucal)$ of square-integrable functions defined below:  
\begin{equation}
  \cS_{l,b}(\Ucal):=\left\{f \in L^1(\Ucal) \cup L^2(\Ucal):\ \ \begin{aligned} &\abs{f(0)}\leq b, ~D^{\boldsymbol{\alpha}}f \mbox{ exists Lebesgue a.e. on } \Ucal ~\forall\boldsymbol{\alpha} \mbox{ s.t. } |\boldsymbol{\alpha}|=l,\\ &\norm{D^{\boldsymbol{\alpha}}f}_{L^2(\Ucal)} \leq b \mbox{ for } \abs{\boldsymbol{\alpha}} \in \{1,l\} \end{aligned}\right\}.  \label{squareintclass}
\end{equation}
The following lemma  states that  functions in  $\cS_{l,b}(\RR^d)$ with sufficient smoothness order $l$ belong to the Barron class. \revised{Its proof essentially follows using arguments from \cite{Barron_1993}, where it was mentioned without explicit quantification. Below, we provide a proof for completeness.}
\begin{lemma}[Smoothness and Barron class]
\label{propsuffcond} 
If $f \in  \cS_{s,b}\left(\RR^d\right)$ for $s:=\lfloor\frac{d}{2}\rfloor+2$, then  we have 
\begin{subequations}
\begin{align}
&  B(f)  \leq b\kappa_d ~\sqrt{d}, \label{bndfourcoeff}\\
& \kappa_d^2:=(d+d^s)\int_{\RR^d}\big(1+\norm{\omega}^{2(s-1)}\big)^{\mspace{-3mu}-1}\dd \omega<\infty.\label{constkappa}
\end{align} %
\end{subequations}
Consequently, $\cS_{s,b}\left(\RR^d\right) \subseteq \cB_{b\kappa_d \sqrt{d} \vee b}$.
\end{lemma}
\begin{proof}
 Since $f \in L^1\left(\RR^d\right) \cup L^2\left(\RR^d\right)$, its Fourier transform $\hat f(\omega)$ exists, and hence, $F(d\omega)=\abs{\hat f(\omega)}d\omega$. Then,  it follows that
    \begin{align}
   & B(f)= \int_{\RR^d} \sup_{x \in \X} \abs{\omega \cdot x}\abs{\hat f(\omega)}d\omega \leq \sqrt{d}\int_{\RR^d} \norm{\omega} \abs{\hat f(\omega)}d\omega, \label{CSapprox}
\end{align}
where we used $\sup_{x \in \X} \abs{\omega \cdot x} \leq \sqrt{d} \norm{\omega}$ which holds by Cauchy-Schwarz inequality.

Next, recall that if the partial derivatives $D^{\boldsymbol{\alpha}}f$, $|\boldsymbol{\alpha}|=s$, exists  on $\mathbb{R}^d$, then all partial derivatives  $D^{\boldsymbol{\alpha}}f$, $0 \leq |\boldsymbol{\alpha}| \leq s$, also exists. Hence, if $\norm{D^{\boldsymbol{\alpha}}f}_{L^2\left(\RR^d\right)} \leq b$ for all $\alpha$ with $\abs{\alpha} \in \{1,s\}$, we have
\begin{align}
 \int_{\RR^d} \norm{\omega} \abs{\hat f(\omega)} d\omega &\stackrel{(a)}{\leq}  \left(\int_{\RR^d}  \frac{d \omega}{1+\norm{\omega}^{2(s-1)}}\right)^{\frac 12} \left(\int_{\RR^d}  \left(\norm{\omega}^2+\norm{\omega}^{2s}\right)\abs{\hat f(\omega)}^2 d\omega\right)^{\frac 12}  \label{bndfcdercoeff}\\
   & \stackrel{(b)}{\leq} \left(\int_{\RR^d}  \frac{d \omega}{1+\norm{\omega}^{2(s-1)}}\right)^{\frac 12} \left(\sum_{\boldsymbol{\alpha}:\abs{\boldsymbol{\alpha}}=1}\norm{D^{\boldsymbol{\alpha}}f}^2_{L^2\left(\RR^d\right)}+\sum_{\boldsymbol{\alpha}:\abs{\boldsymbol{\alpha}}=s}\norm{D^{\boldsymbol{\alpha}}f}^2_{L^2\left(\RR^d\right)}\right)^{\frac 12} \notag \\
  & \stackrel{(c)}{\leq}  \kappa_d b, \label{bndfcfval}
\end{align}
where 
\begin{enumerate}[label = (\alph*),leftmargin=*]
    \item follows from  Cauchy-Schwarz inequality;
    \item is due to Plancherel's theorem;
    \item follows  since $\abs{\{\boldsymbol{\alpha}:\abs{\boldsymbol{\alpha}}=s\}}=d^s$ and  $\norm{D^{\boldsymbol{\alpha}}f}_{L^2\left(\RR^d\right)} \leq b$.  
\end{enumerate}
Combining \eqref{CSapprox} and \eqref{bndfcfval} leads to \eqref{bndfourcoeff}. The final claim follows from \eqref{fourintclas} and \eqref{bndfourcoeff} by noting that $\abs{f(0)} \leq b$ by definition.
\end{proof}

\subsection{Proof of Theorem \ref{supaapprox}} \label{supnormapprox-proof}
\revised{The proof relies on arguments from \cite{Barron-1992} and \cite{Barron_1993}, along with the uniform central limit theorem for uniformly bounded VC function classes.} Fix an arbitrary (small) $\delta>0$, and let  $f:\RR^d \rightarrow \RR$ be such that $\tilde f=f|_\cX$ and $B(f) \vee f(0) \leq c+\delta$. This is possible since  $ c_B^\star(\tilde f)\leq c$. Then, it follows from the proof of \citet[Theorem 2]{Barron_1993} that 
\begin{align}
  f_0(x):=  f(x)- f(0)=\int_{\omega \in \mathbb{R}^d \setminus \{0\}} \varrho(x,\omega)  \mu(d\omega),  \notag
\end{align}
where
\begin{align}
&\varrho(x,\omega)=\frac{B(f)}{\sup_{x \in \X} \abs{\omega \cdot x}}\left(\cos(\omega \cdot x+ \zeta(\omega))-\cos(\zeta(\omega))\right), \notag \\
 & B(f):= \int_{\RR^d} \sup_{x \in \X} \abs{\omega \cdot x}F(d\omega), \notag\\
 & \mu(d\omega)= \frac{\sup_{x \in \X} \abs{\omega \cdot x}F(d\omega)}{B(f)}, \notag 
\end{align}
and $\zeta:\mathbb{R}^d \rightarrow \mathbb{R}$. Note that $\mu \in \mathcal{P}(\mathbb{R}^d)$ is a probability measure. 

Let $\tilde{\Theta}_1\left(k,B(f)\right):=\Theta_{1}(\sqrt{k} \log k,2B(f),0) $ (see \eqref{paramspace}). Then, it further follows from the proofs\footnote{The claims in \citet[Lemma 2- Lemma 4, ~Theorem 3]{Barron_1993}  are stated for  $L_2$ norm, but it is not hard to see from the proof therein that the same also holds for $L^{\infty}$ norm, apart from the following subtlety. In the proof of Lemma 3, it is shown that $\varrho(x,\omega)$, $\omega \in \mathbb{R}^d$, lies in the convex closure of a certain class of step functions, whose  discontinuity points are adjusted to coincide with the continuity points of the underlying measure $\mu$. Similarly, here, the step discontinuities needs to be adjusted to coincide with the continuity points of both $P$ and $Q$. Nevertheless, the same arguments hold since the common continuity points of $P$ and $Q$ form a dense set.} of \citet[Lemma 2-Lemma 4,Theorem 3]{Barron_1993} that  there exists a  probability measure $\mu_k \in \mathcal{P}\left(\tilde{\Theta}_1\left(k,B(f)\right)\right)$ (see \citet[Eqns. (28)-(32)]{Barron_1993}) such that
\begin{align}
\norm{f_0- \int_{\theta \in \tilde{\Theta}_1\left(k,B(f)\right) }  g_{\theta}(\cdot)~   \mu_k\left(d\theta\right)}_{\infty, P,Q} \leq \frac{2 (B(f)+1)}{\sqrt{k}}, \label{approxcc}
\end{align}
 where  $ g_{\theta}(x)=\beta \phi\left(w\cdot x+b\right)$ for $\theta=(\beta,w,b)$. Note that 
$\int_{\tilde{\Theta}_1\left(k,B(f)\right)}\mu_k(d\theta) =1<\infty$.

Next, for each fixed $x$, let  $\upsilon_x:\tilde{\Theta}_1\left(k,B(f)\right) \rightarrow \RR$ be given by $\upsilon_x(\theta):=g_{\theta}(x)$, and consider the function class $\mathcal{V}_k\left(\tilde{\Theta}_1\left(k,B(f)\right)\right)=\left\{\upsilon_x,~x \in \RR^d\right\}$.  
Note that every $ \upsilon_x \in \mathcal{V}_k\left(\tilde{\Theta}_1\left(k,B(f)\right)\right)$ is a composition of an affine function in $\theta$ with the bounded monotonic function  $\beta \phi (\cdot)$. Hence, noting that $\mathcal{V}_k\left(\tilde{\Theta}_1\left(k,B(f)\right)\right)$  is a VC  function class (\cite{AVDV-book}), it follows from 
\citet[Theorem 2.8.3]{AVDV-book} that   it is a uniform Donsker class (in particular, $\mu_k$-Donsker) for all probability measures $\mu \in \mathcal{P}\left(\tilde{\Theta}_1\left(k,B(f)\right)\right)$. 
Furthermore, an application of \citet[Corollary 2.2.8]{AVDV-book}) yields that there exists  $k$ parameter vectors, $\theta_i:=(\beta_i,w_i,b_i) \in \tilde{\Theta}_1\left(k,B(f)\right),~ 1 \leq i \leq k$,
such that (see also \citet[Theorem 2.1]{Yukich-1995}) 
\begin{align}
 \sup_{x \in \mathbb{R}^d}\abs{\int_{\theta \in \tilde{\Theta}_1\left(k,B(f)\right)} g_{\theta}(x)~   \mu_k(d\theta)- \frac 1k \sum_{i=1}^k  g_{\theta_i}(x)} \leq \hat c_d B(f) k^{-\frac 12}, \label{donskerprop}
\end{align}
where $\hat c_d$ is a constant which depends only on $d$. Note that the R.H.S. of \eqref{donskerprop} is independent of $\mu_k$ and depends on $ f$ and $\X$ only via $B(f)$.

From \eqref{approxcc}, \eqref{donskerprop} and triangle inequality, we obtain
\begin{align}
     \norm{ f_0- \frac 1k \sum_{i=1}^k  g_{\theta_i}}_{\infty,P,Q}  \leq \left(\hat c_d B(f)+2 B(f)+2\right)k^{-\frac 12}.\notag
\end{align}
Setting $\theta=\left\{\left\{\left(\frac{\beta_i}{k},w_i,b_i\right)\right\}_{i=1}^k, f(0)\right\}$ and $g_{\theta}(x)= f(0)+\frac 1k \sum_{i=1}^k  g_{\theta_i}(x)$, we have 
\begin{align}
  \norm{f- g_{\theta}}_{\infty,P,Q}  \leq \left((\hat c_d+2) B(f)+2\right)k^{-\frac 12} \leq \left((\hat c_d+2) (c+\delta)+2\right)k^{-\frac 12}. \notag
\end{align}
Next, note that $   \norm{\tilde f- g_{\theta}}_{\infty,P,Q}=  \norm{ f- g_{\theta}}_{\infty,P,Q}$ and $ g_{\theta} \in \mathcal{G}_k^*\left(B(f) \vee f(0)\right) \subseteq \mathcal{G}_k^*\left(c+\delta\right)
$. Since  $\delta>0$ is arbitrary, we obtain that there exists $ g_{\theta} \in  \mathcal{G}_k^*\left(c\right)
$
\begin{align}
  \norm{\tilde f- g_{\theta}}_{\infty,P,Q}   \leq \left((\hat c_d+2) c+2\right)k^{-\frac 12} =:  \tilde C_{d,c}~k^{-\frac 12}, \label{finaapperrbarcls}
\end{align}
thus proving the  claim in \eqref{approxratefin}.

\revised{On the other hand, it follows similar to \eqref{bndfcdercoeff} in Lemma \ref{propsuffcond} that for a fixed $\epsilon>0$ and $l(\epsilon)=d/2+1+\epsilon$, the set of functions $f \in \RR^d \rightarrow \RR$  such that $B(f) \leq c$ includes those whose Fourier transform $\hat f(\omega)$ satisfies
\begin{align}
   \int_{\RR^d}  \left(\norm{\omega}^2+\norm{\omega}^{2l(\epsilon)}\right)\abs{\hat f(\omega)}^2 d\omega \leq c^2 d^{-1}\left(\int_{\RR^d}  \frac{d \omega}{1+\norm{\omega}^{2(l(\epsilon)-1)}}\right)^{-1},
\end{align}
since $\int_{\RR^d}  \frac{d \omega}{1+\norm{\omega}^{2(l(\epsilon)-1)}}<\infty$.  Then, \eqref{approxratelb} follows from the proof of \cite{Barron-1992}[Theorem 3]. Note from the proof therein that the constant in \eqref{approxratelb} may in general depend on $d$ and $\epsilon$.}
\subsection{Proof of Corollary \ref{cor:bndfourcoeff}}\label{cor:bndfourcoeff-proof}
By Theorem \ref{supaapprox}, it suffices to show that there exists an extension $f_{\mathsf{e}}$ of $f$ from $\Ucal$ to $\RR^d$ such that $B(f_{\mathsf{e}}) \vee f_{\mathsf{e}}(0) \leq \bar c_{b,c,d} $. 
Let $\boldsymbol{\alpha}_j$ denote a multi-index of order $j$, and recall that $s:=\lfloor\frac{d}{2}\rfloor+2$.  Consider an extension of $D^{\boldsymbol{\alpha}_s}f$ from $\Ucal$ to $\RR^d$ for each $\boldsymbol{\alpha}_s$ as follows:
\begin{align}
    D^{\boldsymbol{\alpha}_s}f(x):=\inf_{x' \in \Ucal} D^{\boldsymbol{\alpha}_s}f(x')+c\norm{x-x'}^{\delta}, ~x \in \RR^d \setminus \Ucal. \label{holdcontext}
\end{align}
Note that $D^{\boldsymbol{\alpha}_s}f$ extended this way  is  H\"{o}lder continuous with the same constant $c$ and exponent $\delta$ on $\RR^d$. 
Fixing $D^{\boldsymbol{\alpha}_s}f$ on $\RR^d$ induces an extension of all lower (and also higher) order derivatives $D^{\boldsymbol{\alpha}_j}f,~0 \leq j<s$  to $\RR^d$,  which can be defined recursively as $D^{\boldsymbol{\alpha}_1}D^{\boldsymbol{\alpha}_{s-j}}f(x)=D^{\boldsymbol{\alpha}_1+\boldsymbol{\alpha}_{s-j}}f(x)$, $x \in \RR^d$, for all $\boldsymbol{\alpha}_1$, $\boldsymbol{\alpha}_{s-j}$ and  $j=1, \ldots, s$.

 Let  $\Ucal':=\{x' \in \RR^d: \norm{x'-x} < 1 \mbox{ for some }x \in \X\}$. Suppose $\Ucal \subset \Ucal'$. 
By the mean value theorem, we have for any $x,x' \in \Ucal'$ and $j=1,\ldots,s$,
\begin{align}
  \abs{D^{\boldsymbol{\alpha}_{s-j}}f(x')} & \leq  \abs{D^{\boldsymbol{\alpha}_{s-j}}f(x)}+\max_{\substack{\tilde x \in \Ucal',\\\boldsymbol{\alpha}_1}}\abs{D^{\boldsymbol{\alpha}_{s-j}+\boldsymbol{\alpha}_1}f(\tilde x)} \norm{x-x'}_1 \notag \\
    &\leq  \abs{D^{\boldsymbol{\alpha}_{s-j}}f(x)}+\max_{\substack{\tilde x \in \Ucal',\\\boldsymbol{\alpha}_1}}\abs{D^{\boldsymbol{\alpha}_{s-j}+\boldsymbol{\alpha}_1}f(\tilde x)} \sqrt{d} \norm{x-x'}, \label{recuderval}
\end{align}
where the last step follows from $\norm{x-x'}_1 \leq \sqrt{d} \norm{x-x'}$.  Also, note from \eqref{holdcontext} that 
$ D^{\boldsymbol{\alpha}_s}f(x) < b+c$ for all $x \in \Ucal'$, and recall that since $f \in \cH_{b,c}^{s,\delta}(\Ucal)$, we have $ \abs{D^{\boldsymbol{\alpha}_{s-j}}f(x)} \leq b$ for all $x \in \Ucal$. Then, for any $x' \in \Ucal'$,  taking $x \in \X$  satisfying $\norm{x-x'}\leq 1$ (such an $x$ exists by definition of $\Ucal'$) in \eqref{recuderval} yields 
\begin{align}
 \abs{D^{\boldsymbol{\alpha}_{s-1}}f(x')} \leq b+(b+c)\sqrt{d}.  \label{startrecur}  
\end{align}
Starting from \eqref{startrecur} and  recursively applying  \eqref{recuderval}, we obtain for $j=1,\ldots,s$, and $x' \in \Ucal'$,
\begin{align}
   \abs{D^{\boldsymbol{\alpha}_{s-j}}f(x')} &\leq b \sum_{i=1}^j d^{\frac{i-1}{2}} +(b+c)d^{\frac{j}{2}} \leq b\frac{1-d^{\frac{s}{2}}}{1-\sqrt{d}}+(b+c)d^{\frac{s}{2}}=: \tilde b. \label{bndderallpar}
\end{align}
 Thus, the extension $f$ from $\Ucal$ to $\RR^d$ satisfies $f|_{\Ucal'} \in \cH_{\tilde b,c}^{s,\delta}(\Ucal')$. If $\Ucal' \subseteq \Ucal$, then $f|_{\Ucal'} \in \cH_{b,c}^{s,\delta}(\Ucal')$ by definition, and thus, in either case, $f|_{\Ucal'} \in \cH_{\tilde b,c}^{s,\delta}(\Ucal')$.
 
 The desired final extension is $f_{\mathsf{e}}:\RR^d \rightarrow \RR$ given by $f_{\mathsf{e}}(x):=f(x)\cdot f_{\mathsf{C}}(x)$, where 
 \begin{align}
   &  f_{\mathsf{C}}(x):=\ind_{\X'} \ast \psi_{\frac 12} (x):=\int_{\RR^d} \ind_{\X'}(y) \psi_{\frac 12}(x-y) dy,~x\in \RR^d,\label{cutofffndef} \\
   & \X':=\left\{x' \in \RR^d: \norm{x'-x} \leq 0.5 \mbox{ for some }x \in \X\right\},\notag  \\
   & \psi (x):=\begin{cases} u^{-1}e^{-\frac{1}{\frac 12-\norm{x}^2}},~& \norm{x} <\frac{1}{2}, \\0,& \mbox{otherwise}, \end{cases}
 \end{align}
 and $u$ is the normalization constant such that $\int_{\RR^d} \psi (x) dx=1$. Note that $\psi \in \mathsf{C}^{\infty}\left(\RR^d\right)$, and consequently, $f_{\mathsf{C}} \in \mathsf{C}^{\infty}\left(\RR^d\right)$ from \eqref{cutofffndef} by dominated convergence theorem. Also, observe that $f_{\mathsf{C}}(x)=1$ for $x \in \X$, $f_{\mathsf{C}}(x)=0$ for $x \in \RR^d \setminus \Ucal'$ and $f_{\mathsf{C}}(x)  \in (0,1)$ for $x \in  \Ucal'  \setminus \X$. Hence, $f_{\mathsf{e}}(x)=f(x)$ for $x \in \X$, $f_{\mathsf{e}}(x)=0$ for $x \in \RR^d \setminus \Ucal'$ and $\abs{f_{\mathsf{e}}(x)} \leq  \abs{f(x)}$ for $x \in  \Ucal'  \setminus \X$, thus satisfying $f_{\mathsf{e}}|_\cX=f|_\cX=\tilde f$ as required. Moroever, for all $j=0,\ldots,s$,
 \begin{subequations}\label{derbndinu} 
 \begin{equation}\label{derbndinu1} 
    \abs{D^{\boldsymbol{\alpha_j}}f_{\mathsf{e}}(x)} \stackrel{(a)}{\leq}  2^j\tilde b \max_{\substack{x \in \Ucal',\\\boldsymbol{\alpha}:\abs{\boldsymbol{\alpha}} \leq j }} \abs{D^{\boldsymbol{\alpha}}f_{\mathsf{C}}(x)}  \stackrel{(b)}{\leq} 2^s\tilde b \max_{\substack{x:\norm{x} \leq 0.5,\\\boldsymbol{\alpha}:\abs{\boldsymbol{\alpha}}\leq s }} \abs{D^{\boldsymbol{\alpha}}\psi(x)}=:\hat b,~x\in \Ucal', 
 \end{equation} 
  \begin{equation}\label{derbndinu2} 
   D^{\boldsymbol{\alpha_j}}f_{\mathsf{e}}(x)=0,~ x\notin \Ucal',
 \end{equation} 
 \end{subequations}
 where 
\begin{enumerate}[label = (\alph*),leftmargin=*]
     \item follows using chain rule for differentiation and \eqref{bndderallpar};
     \item follows from the definition in \eqref{cutofffndef}.
 \end{enumerate}
Then,  we have for $j=0,\ldots,s$,
\begin{align}
\norm{D^{\boldsymbol{\alpha_j}}f_{\mathsf{e}}}_{L^2\left(\RR^d\right)}^2  &=\int_{\RR^d} (D^{\boldsymbol{\alpha_j}}f_{\mathsf{e}})^2(x) dx \notag \\
&=\int_{\Ucal'} (D^{\boldsymbol{\alpha_j}}f_{\mathsf{e}})^2(x) dx \leq \hat b^2~ \mathsf{Vol}_d(0.5\sqrt{d}+1)\notag \\
&=\hat b^2 \frac{\pi^{\frac{d}{2}}}{\Gamma(\frac{d}{2}+1)} (0.5\sqrt{d}+1)^d=:\bar b, \label{allderivsqint}
\end{align}
where $\mathsf{Vol}_d(r)$ denotes the volume of a Euclidean ball in $\RR^d$ with radius $r$ and $\Gamma$ denotes the gamma function.
 Defining $b':=\sqrt{\bar b}$ and noting that $b' \geq \hat b$, we have from \eqref{derbndinu} and \eqref{allderivsqint} that $ f_{\mathsf{e}}(x) \in \tilde{\cS}_{s,b'}\left(\RR^d\right) $, where  
\begin{equation}
  \tilde{\cS}_{s,b'}\left(\RR^d\right):=\left\{f \in L^1\left(\RR^d\right) \cup L^2\left(\RR^d\right):\ \ \begin{aligned} &\abs{f(0)}\leq b', ~D^{\boldsymbol{\alpha}}f \mbox{ exists Lebesgue a.e. on } \RR^d ~\forall\boldsymbol{\alpha} \mbox{ s.t. } |\boldsymbol{\alpha}|=s,\\ &\norm{D^{\boldsymbol{\alpha}}f}_{L^2\left(\RR^d\right)} \leq b' \mbox{ for } \abs{\boldsymbol{\alpha}} = 1,\ldots,s \end{aligned}\right\}.  \label{squareintclassallder}
\end{equation}
 Observe that $\tilde{\cS}_{s,b'}\left(\RR^d\right) \subseteq \cS_{s,b'}\left(\RR^d\right)$ (see \eqref{squareintclass}). This  implies via Lemma  \ref{propsuffcond} that $B(f_{\mathsf{e}}) \leq c':=\kappa_d \sqrt{d}\mspace{2 mu}b'$ and 
\begin{align}
    f_{\mathsf{e}} \in \cB_{ b'   \vee c'} \cap  \tilde{\cS}_{s,b'}\left(\RR^d\right) \subseteq \cB_{ b'   \vee c'} \cap  \cS_{s,b'}\left(\RR^d\right). \label{finextchar}
\end{align}
 Then, by defining
 \begin{align}
     \bar c_{b,c,d}&:=b' \vee c', 
     \label{constapproxhold}
 \end{align}
 where 
  \begin{align}
    &b'=\pi^{\frac{d}{4}}\Gamma^{-1/2}(0.5d+1) (0.5\sqrt{d}+1)^{\frac{d}{2}}2^s \left(b\frac{1-d^{\frac{s}{2}}}{1-\sqrt{d}}+(b+c)d^{\frac{s}{2}}\right) \max_{\substack{x:\norm{x} \leq 0.5,\\\boldsymbol{\alpha}:\abs{\boldsymbol{\alpha}}\leq s }} \psi^{\boldsymbol{(\alpha)}}(x) ,\label{fourcoeffholdval1} \\
     &c'=\sqrt{d}\kappa_d b' , \label{fourcoeffholdval2} \\
  & \kappa_d^2=  (d+d^s)\int_{\RR^d}\big(1+\norm{\omega}^{2(s-1)}\big)^{\mspace{-3mu}-1}\dd \omega, \notag
\end{align}
it  follows from Theorem \ref{supaapprox} (see \eqref{finaapperrbarcls}) that there exists $g \in \mathcal{G}_k^*\left(\bar c_{b,c,d}\right)$ such~that
\begin{align}
\big\|\tilde f-g\big\|_{\infty, P,Q} \leq \tilde C_{d,\bar c_{b,c,d}} ~k^{-\frac 12}.
\end{align}
This completes the proof.
\subsection{Proof of Theorem \ref{empesterrbnd}} \label{empesterrbnd-proof}
We will show that Theorem \ref{empesterrbnd} holds with 
\begin{align}
& V_{k,\mathbf{a},\gamma}:= 4Ca_2^2 \mspace{1 mu} k \mspace{1 mu} R_{k,\mathbf{a},\gamma}^2, \label{Vkconstdef} \\
   & E_{k,\mathbf{a},n,\gamma}:= 
    2 \sqrt{2}n^{-\frac 12}k a_2 R_{k,\mathbf{a},\gamma}=4 \sqrt{2}n^{-\frac 12}k^{3/2} a_2 \left(\bar{\gamma}'_{\mathcal{G}_{k}(\mathbf{a})}+1\right), \label{Ekconstdef}
   \end{align}
   where
\begin{align}
   &R_{k,\mathbf{a},\gamma}:= 2\left(\bar{\gamma}'_{\mathcal{G}_{k}(\mathbf{a})}+1\right)\sqrt{k},
\end{align}
and $\bar{\gamma}'_{\mathcal{G}_{k}(\mathbf{a})}$ is defined in \eqref{maxdergamma}.
We have
\begin{flalign}
&\hat{\mathsf{H}}_{\gamma,\mathcal{G}_k(\mathbf{a})}(x^n,y^n)-\mathsf{H}_{\gamma, \mathcal{G}_k(\mathbf{a})}(P,Q) \notag\\
&= \sup_{g_{\theta} \in \mathcal{G}_k(\mathbf{a})} \frac 1n \sum_{i=1}^n g_{\theta}(x_i)-\frac 1n \sum_{i=1}^n \gamma(g_{\theta}(y_i))-\left(\sup_{g_{\theta} \in \mathcal{G}_k(\mathbf{a})}\mathbb{E}_P[g_{\theta}(X)]-\mathbb{E}_Q\left[\gamma(g_{\theta}(Y))\right]\right) \notag \\
    &\leq   \sup_{g_{\theta} \in \mathcal{G}_k(\mathbf{a})} \frac 1n \sum_{i=1}^n g_{\theta}(x_i)-\frac 1n \sum_{i=1}^n \gamma(g_{\theta}(y_i))-\mathbb{E}_P[g_{\theta}(X)]+\mathbb{E}_Q\left[\gamma(g_{\theta}(Y))\right]. \label{supdiffer}&&
\end{flalign}
Let 
\begin{align}
Z_{\theta}:=\frac 1n \sum_{i=1}^n g_{\theta}(X_i)-\frac 1n \sum_{i=1}^n \gamma\left(g_{\theta}(Y_i)\right)-\mathbb{E}_P[g_{\theta}(X)]+\mathbb{E}_Q\left[\gamma(g_{\theta}(Y))\right]. \label{defnztheta}
\end{align}
We have
\begin{flalign}
 & \abs{   Z_{\theta}-  Z_{\theta'}} \leq  \sum_{i=1}^n \frac 1n \abs{g_{\theta}(X_i)-g_{\theta'}(X_i)-\mathbb{E}_P[g_{\theta}(X)-g_{\theta'}(X)]} \notag \\
 &\qquad \qquad \qquad \qquad +  \frac 1n \abs{\gamma(g_{\theta}(Y_i))-\gamma(g_{\theta'}(Y_i))-\mathbb{E}_Q\left[\gamma(g_{\theta}(Y))-\gamma(g_{\theta'}(Y))\right]}. \label{sumzthet} &&
\end{flalign}
 Since $0 \leq \phi(x)\leq 1$ for all $x \in \mathbb{R}^d$,  for any $x,x' \in \mathcal{X}$ and $\theta=\left(\{\beta_i,w_i,b_i\}_{i=1}^k,b_0\right),\theta'=\left(\{\beta'_i,w'_i,b'_i\}_{i=1}^k,b'_0\right) \in \Theta_{k}(\mathbf{a})$,  
 \begin{equation}
  \abs{g_{\theta}(x)-g_{\theta'}(x')} \leq \sum_{i=1}^k \abs{\beta_i-\beta_i'} \leq  \norm{\boldsymbol{\beta}(\theta)-\boldsymbol{\beta}(\theta')}_1,\label{difffngk}
   \end{equation}
   where $\boldsymbol{\beta}(\theta):=(\beta_1,\ldots,\beta_k)$.   Moreover, an application of the mean value theorem yields that for all $\theta,\theta' \in \Theta_{k}(\mathbf{a})$,
    \begin{equation}
   \abs{\gamma(g_{\theta}(x))-\gamma(g_{\theta'}(x'))} \leq \bar{\gamma}'_{\mathcal{G}_{k}(\mathbf{a})} \abs{g_{\theta}(x)-g_{\theta'}(x')} \leq \bar{\gamma}'_{\mathcal{G}_{k}(\mathbf{a})}   \norm{\boldsymbol{\beta}(\theta)-\boldsymbol{\beta}(\theta')}_1,  \label{diffgamfngk}
    \end{equation} 
 where $\bar{\gamma}'_{\mathcal{G}_{k}(\mathbf{a})}$ is defined in \eqref{maxdergamma}.
Hence, with probability one
\begin{flalign}
& \frac 1n \abs{g_{\theta}(X_i)-g_{\theta'}(X_i)-\mathbb{E}_P[g_{\theta}(X_i)-g_{\theta'}(X_i)]}+ \frac 1n \big|\gamma(g_{\theta}(Y_i))-\gamma(g_{\theta'}(Y_i))  -\mathbb{E}_Q\left[\gamma(g_{\theta}(Y_i))-\gamma(g_{\theta'}(Y_i))\right]\big| \notag \\
&  \leq \frac 1n \left[ \abs{g_{\theta}(X_i)-g_{\theta'}(X_i)}+ \abs{\mathbb{E}_P[g_{\theta}(X_i)-g_{\theta'}(X_i)]} + \abs{\gamma(g_{\theta}(Y_i))-\gamma(g_{\theta'}(Y_i))} + \abs{\mathbb{E}_Q\left[\gamma(g_{\theta}(Y_i))-\gamma(g_{\theta'}(Y_i))\right]} \right]\notag \\
 &\leq  \frac{1}{n} s_{k,\mathbf{a},\gamma} \norm{\boldsymbol{\beta}(\theta)-\boldsymbol{\beta}(\theta')}_1, \label{bndtermstv} &&
\end{flalign}
where $s_{k,\mathbf{a},\gamma}:=2\left(\bar{\gamma}'_{\mathcal{G}_{k}(\mathbf{a})}+1\right)$. Note that $\mathbb{E}\left[Z_{\theta}\right]=0$ for all $\theta \in \Theta_{k}(\mathbf{a})$. Then, using the fact that $\norm{\boldsymbol{\beta}(\theta)-\boldsymbol{\beta}(\theta')}_1 \leq \sqrt{k}\norm{\boldsymbol{\beta}(\theta)-\boldsymbol{\beta}(\theta')}$, it follows from \eqref{sumzthet} and \eqref{bndtermstv} via Hoeffding's lemma that 
\begin{align}
    \mathbb{E}\left[e^{t\left(Z_{\theta}-Z_{\theta'}\right)} \right]\leq   e^{\frac{1}{2}t^2 \mathsf{d}_{k,\mathbf{a},n,\gamma}(\theta,\theta')^2}, 
\end{align}
 where
\begin{align}
    \mathsf{d}_{k,\mathbf{a},n,\gamma}(\theta,\theta'):= \frac{s_{k,\mathbf{a},\gamma}\sqrt{k}\norm{\boldsymbol{\beta}(\theta)-\boldsymbol{\beta}(\theta')}}{\sqrt{n}}:=\frac{R_{k,\mathbf{a},\gamma}}{\sqrt{n}}\norm{\boldsymbol{\beta}(\theta)-\boldsymbol{\beta}(\theta')}.
\end{align}
It follows that $\{Z_{\theta}\}_{\theta \in \Theta_{k}(\mathbf{a})}$ is a separable subgaussian process on the metric space $( \Theta_{k}(\mathbf{a}), \mathsf{d}_{k,\mathbf{a},n,\gamma}(\theta,\theta'))$.
 Next, note that  $ N\left(\Theta_{k}(\mathbf{a}), \mathsf{d}_{k,\mathbf{a},n,\gamma}(\cdot,\cdot),\epsilon\right)=N\left([-a_{2},a_{2}]^k, n^{-\frac 12}R_{k,\mathbf{a},\gamma}\norm{\cdot},\epsilon\right)$. Also, $[-a_{2},a_{2}]^k \subseteq B^k\left(\sqrt{k} ~a_{2}\right)$. Hence, we have
\begin{flalign}
    N\left(\Theta_{k}(\mathbf{a}), \mathsf{d}_{k,\mathbf{a},n,\gamma}(\cdot,\cdot),\epsilon\right) &\leq N\left(B^k\left(\sqrt{k} ~a_{2}\right), n^{-\frac 12}R_{k,\mathbf{a},\gamma}\norm{\cdot},\epsilon\right) \notag \\ 
    &= N\left(B^k\left(\sqrt{k} ~a_{2}\right), \norm{\cdot},\sqrt{n}R_{k,\mathbf{a},\gamma}^{-1}\epsilon\right) \notag \\
    & \leq \frac{\left(\sqrt{k} ~a_{2}+\sqrt{n} R_{k,\mathbf{a},\gamma}^{-1}\epsilon\right)^{k}}{\left(\sqrt{n} R_{k,\mathbf{a},\gamma}^{-1}\epsilon\right)^{k}} \label{applycovfo}\\
    &=\left(1+\frac{\sqrt{k} ~a_{2}~R_{k,\mathbf{a},\gamma}}{\sqrt{n} \epsilon}\right)^{k}, \notag &&
\end{flalign}
where, in \eqref{applycovfo}, we used that the covering number of Euclidean ball $B^d(r)$ w.r.t. Euclidean norm  satisfies
\begin{align}
   N\left(B^{d}(r),\norm{\cdot},\epsilon\right) \leq \left(\frac{r+\epsilon}{\epsilon}\right)^d.
\end{align}
Also,  for $\epsilon \geq \mathsf{diam}\left(\Theta_{k}(\mathbf{a}),\mathsf{d}_{k,\mathbf{a},n,\gamma}\right):=\max_{\theta,\theta' \in \Theta_{k}(\mathbf{a})}\mathsf{d}_{k,\mathbf{a},n,\gamma}(\theta,\theta')= 2\sqrt{k} \mspace{1 mu}a_{2} R_{k,\mathbf{a},\gamma}n^{-\frac 12}$, we have that $N\left(\Theta_{k}(\mathbf{a}), \mathsf{d}_{k,\mathbf{a},n,\gamma}(\cdot,\cdot),\epsilon\right)=1$. Then, 
\begin{flalign}
E_{k,\mathbf{a},n,\gamma}&:=\int_{0}^{\infty} \sqrt{\log N\left(\Theta_{k}(\mathbf{a}), \mathsf{d}_{k,\mathbf{a},n,\gamma}(\cdot,\cdot),\epsilon\right)}d\epsilon \notag \\
&= \int_{0}^{\mathsf{diam}\left(\Theta_{k}(\mathbf{a}),\mathsf{d}_{k,\mathbf{a},n,\gamma}\right)} \sqrt{\log N\left(\Theta_{k}(\mathbf{a}), \mathsf{d}_{k,\mathbf{a},n,\gamma}(\cdot,\cdot),\epsilon\right)}d\epsilon \notag \\
    & \leq  \sqrt{k}\int_{0}^{\mathsf{diam}\left(\Theta_{k}(\mathbf{a}),\mathsf{d}_{k,\mathbf{a},n,\gamma}\right)} \sqrt{\log \left(1+\frac{a_{2}\sqrt{k}R_{k,\mathbf{a},\gamma}}{\sqrt{n}\epsilon}\right)}d\epsilon \notag \\
    & \leq n^{-\frac{1}{4}} k^{\frac 34} \sqrt{a_{2} R_{k,\mathbf{a},\gamma}}~\int_{0}^{\mathsf{diam}\left(\Theta_{k}(\mathbf{a}),\mathsf{d}_{k,\mathbf{a},n,\gamma}\right)} \epsilon^{-\frac 12 }d\epsilon \label{applylogupbnd}\\
    &=2 k^{\frac 34} n^{-\frac{1}{4}}\sqrt{a_{2} R_{k,\mathbf{a},\gamma}~\mathsf{diam}\left(\Theta_{k}(\mathbf{a}),\mathsf{d}_{k,\mathbf{a},n,\gamma}\right)}, \label{entropyintcondfin} &&
\end{flalign}
where, we used the inequality $\log(1+x) \leq x$ (for $x \geq -1$) in \eqref{applylogupbnd}. It follows from Theorem \ref{thm:tailineq} that there exists a constant $C$ such that for $\delta >0$,
\begin{flalign}
&\mathbb{P}\left(\sup_{g_{\theta} \in \mathcal{G}_k(\mathbf{a})}Z_{\theta} \geq CE_{k,\mathbf{a},n,\gamma}+\delta\right)  = \mathbb{P}\left(\sup_{g_{\theta} \in \mathcal{G}_k(\mathbf{a})} Z_{\theta}-Z_{\mathbf{0}} \geq CE_{k,\mathbf{a},n,\gamma}+ \delta\right) \notag \\
&\leq Ce^{-\frac{\delta^2}{C\mathsf{diam}\left(\Theta_{k}(\mathbf{a}),\mathsf{d}_{k,\mathbf{a},n,\gamma}\right)^2}}=C e^{-\frac{n\delta^2}{4Ca_2^2 R_{k,\mathbf{a},\gamma}^2 k}}, \label{onesiddevineq1} &&
\end{flalign}
where $Z_{\mathbf{0}}=0$. It follows similarly that for $\delta >0$,
\begin{align}
\mathbb{P}\left(\sup_{g_{\theta} \in \mathcal{G}_k(\mathbf{a})}-Z_{\theta} \geq \delta+CE_{k,\mathbf{a},n,\gamma}\right) \leq C e^{-\frac{n \delta^2}{4Ca_2^2 R_{k,\mathbf{a},\gamma}^2 k}}.\label{onesiddevineq2}
\end{align}
Combining \eqref{onesiddevineq1} and \eqref{onesiddevineq2} yields
\begin{align}
\mathbb{P}\left(\sup_{g_{\theta} \in \mathcal{G}_k(\mathbf{a})}|Z_{\theta}| \geq \delta+CE_{k,\mathbf{a},n,\gamma}\right) \leq 2C e^{-\frac{n\delta^2}{4Ca_2^2 R_{k,\mathbf{a},\gamma}^2 k}}. \label{magbndsup}
\end{align}
From \eqref{supdiffer}, \eqref{defnztheta} and   \eqref{magbndsup}, we obtain  that for $\delta >0$,
\begin{flalign}
   & \mathbb{P}\left(\abs{\mathsf{H}_{\gamma, \mathcal{G}_k(\mathbf{a})}(P,Q)-\hat{\mathsf{H}}_{\gamma,\mathcal{G}_k(\mathbf{a})}(X^n,Y^n)} \geq \delta+CE_{k,\mathbf{a},n,\gamma} \right) \notag \\
    &\leq \mathbb{P}\left(\sup_{g_{\theta} \in \mathcal{G}_k(\mathbf{a})}|Z_{\theta}| \geq \delta+CE_{k,\mathbf{a},n,\gamma}\right) \leq 2C e^{-\frac{n\delta^2}{4Ca_2^2 R_{k,\mathbf{a},\gamma}^2 k}}. \label{devbnddelt} &&
\end{flalign}
\section{Appendix: KL divergence} \label{KLdivresults}
\subsection{Proof of Theorem  \ref{strongcons}}  \label{strongcons-proof}
Let $D_{\mathcal{G}_k(\mathbf{a}_k)}(P,Q):=\mathsf{H}_{\gamma_{\mathsf{KL}}, \mathcal{G}_k(\mathbf{a}_k)}(P,Q)$.
The proof of Theorem  \ref{strongcons} relies on the following lemma, whose proof is given in Appendix \ref{lem:consicomp-proof}.
\begin{lemma} \label{lem:consicomp}
Let  $P,Q\in \mathcal{P}_{\mathsf{KL}}(\X)$. Then, for $X^n \sim P^{\otimes n}$ and $Y^n \sim Q^{\otimes n}$, the following holds for any $\alpha>0$:
\end{lemma}
\begin{enumerate}[label = (\roman*),leftmargin=*]
  \item For  $n, k_n,\mathbf{a}_{k_n}=(a_{1,k_n},a_{2,k_n},a_{3,k_n})$  such that   $k_n^{\frac 32}a_{2,k_n}e^{k_na_{2,k_n}+a_{3,k_n}} = O \left(n^{\frac{1-\alpha}{2}}\right)$,
\begin{align}
  \hat{D}_{\mathcal{G}_{k_n}\left(\mathbf{a}_{k_n}\right)}(X^n,Y^n)  \xrightarrow[n\rightarrow \infty]{}    D_{\mathcal{G}_{k_n}\left(\mathbf{a}_{k_n}\right)}(P,Q),\quad  \mathbb{P}-\mbox{a.s.}.\label{errbndest}
\end{align}
\item  For $n, k,\mathbf{a}_k=(a_{1,k},a_{2,k},a_{3,k})$  such that   $k^{\frac 32}a_{2,k}e^{ka_{2,k}+a_{3,k}} = O \left(n^{\frac{1-\alpha}{2}}\right)$
 \begin{align}
     \mathbb{E}\left[\abs{\hat{D}_{\mathcal{G}_{k}\left(\mathbf{a}_k\right)}(X^n,Y^n)-D_{\mathcal{G}_k(\mathbf{a}_k)}(P,Q)}\right] =O\left(n^{-\frac 12}k^{\frac 32}a_{2,k} e^{ka_{2,k}+a_{3,k}}\right). \label{kldivgempesterr}
 \end{align}
 \end{enumerate}
We proceed to prove \eqref{finbndascon}.  Since  $f_{\mathsf{KL}} \in \mathsf{C}\left(\X\right)$ for a compact set $\X$, it follows from \citet[Theorem 2.8]{stinchcombe1990approximating} that for any $\epsilon>0$ and $k \geq k_0(\epsilon)$, there exists a $g_{\tilde \theta} \in \mathcal{G}_k(\mathbf{1})$ such that 
  \begin{align}
      \sup_{x \in \X}\abs{f_{\mathsf{KL}}(x)-g_{\tilde \theta}(x)} \leq \epsilon. \label{approxbndwt}
  \end{align}
  This implies that 
  \begin{align}
     \lim_{k \rightarrow \infty} D_{\mathcal{G}_k(\mathbf{1})}(P,Q)  =\kl{P}{Q}. \label{approxlim}
  \end{align}
To see this, note that
\begin{align}
 D_{\mathcal{G}_k(\mathbf{1})}(P,Q) \leq  \kl{P}{Q},~ \forall  k \in \mathbb{N}, \label{limitapprkl}
\end{align}
 by \eqref{CC-charact} since $g_{\theta}$ is continuous and  bounded ($\abs{g_{\theta} } \leq k+1$). Moreover, the left hand side (L.H.S.) of \eqref{limitapprkl} is monotonically increasing in $k$, and being bounded, has a limit point. Then, \eqref{approxlim} will follow if we show that the limit point is $\kl{P}{Q}$. Assume otherwise that $\lim_{k \rightarrow \infty}   D_{\mathcal{G}_k(\mathbf{1})}(P,Q)<\kl{P}{Q}$. Note that $\mathcal{G}_k(\mathbf{1})$ is a closed set and hence the supremum in the variational form of the L.H.S. of \eqref{limitapprkl} is a maximum. Then, defining 
 \begin{align}
  D(g):=1+\mathbb{E}_P[g(X)]-\mathbb{E}_Q\left[e^{g(Y)}\right],   
 \end{align}
this  implies that there exists $\delta>0$ and
\begin{align}
    g_{\theta_{k}^*}:= \argmax_{g_{\theta} \in \mathcal{G}_k(\mathbf{1})} D(g_{\theta}), \label{optkgtheta}
\end{align}
such that for all $k$,
\begin{align}
 \kl{P}{Q}- D(g_{\theta_{k}^*}) \geq \delta. \label{contradineqkl}
\end{align}
However, it follows from \eqref{approxbndwt} that for all $k \geq k_0(\epsilon)$,
\begin{flalign}
   \kl{P}{Q}- D(g_{\theta_{k}^*}) &\leq \kl{P}{Q}- D(g_{\tilde \theta}) \notag\\
   &\leq \mathbb{E}_P\left[\abs{f_{\mathsf{KL}}(X)- g_{\tilde  \theta}(X)}\right]+\mathbb{E}_Q\left[\abs{e^{f_{\mathsf{KL}}(Y)}-e^{g_{\tilde \theta}(Y)}}\right] \notag\\
     &\leq  \mathbb{E}_P\left[\abs{f_{\mathsf{KL}}(X)- g_{\tilde  \theta}(X)}\right]+L_{P,Q}~\mathbb{E}_Q\left[\abs{1-e^{g_{\tilde  \theta}(Y)-f_{\mathsf{KL}}(Y)}}\right] \label{intrconsteps} \\
  &\leq  \epsilon+  L_{P,Q}(e^{\epsilon}-1), \label{bndtotdiffkl}&&  
\end{flalign}
where \eqref{bndtotdiffkl} follows from  \eqref{approxbndwt}. Note that
\begin{align}
 0 \leq   L_{P,Q}:= \norm{\frac{\dd P}{\dd Q}}_{\infty}<\infty, \label{lmaxdefn}
\end{align}
since $e^{f_{\mathsf{KL}}}$ is a continuous function and hence bounded over a compact support $\X$.
 Taking $\epsilon$ sufficiently small in \eqref{bndtotdiffkl} contradicts \eqref{contradineqkl}, thus proving \eqref{approxlim}. 
 Next, for $a_{3,k}=a_{2,k}=a_{1,k}=1$ and any $\eta>0$, $k^{\frac 32}a_{2,k}e^{ka_{2,k}+a_{3,k}} < e^{k(1+\eta)}$  provided  $k$ is sufficiently large. Then, \eqref{finbndascon} follows from \eqref{errbndest}  and \eqref{approxlim} by letting $k=k_n  \rightarrow \infty$ (subject to constraint in Lemma \ref{lem:consicomp}$(i)$), and  noting that $\eta>0$ is arbitrary.

Next, we prove \eqref{KLeffbndsimp}. 
Note that since $f_{\mathsf{KL}} \in \mathcal{I}(M)$, we have from \eqref{finaapperrbarcls}  that for $k$ such that $m_k \geq M$, there exists $g_{\theta} \in \mathcal{G}_{k}^*(m_k)$ satisfying 
 \begin{align}
      \norm{ f_{\mathsf{KL}}-  g_{\theta}}_{\infty,P,Q} \leq \tilde C_{d,M} k^{-\frac 12 }=\left((\hat c_d+2) M+2\right)k^{-\frac 12}. \notag 
 \end{align}
 On the other hand, for $k$ such that $m_k < M$, taking $g_{\mathbf{0}}=0$ yields $ \norm{ f_{\mathsf{KL}}-  g_{\mathbf{0}}}_{\infty,P,Q} \leq M$. Hence, for all $k$,  there exists $g_{\theta_k^*} \in \mathcal{G}_{k}^*(m_k) $ such that
 \begin{align}
     \norm{ f_{\mathsf{KL}}-  g_{\theta_k^*}}_{\infty,P,Q} \leq D_{d,M,\mathbf{m}}k^{-\frac 12}, \label{newconstdimdef}
 \end{align}
 where $\mathbf{m}=\{m_k\}_{k \in \NN}$,
 \begin{align}
    & D_{d,M,\mathbf{m}}:=\tilde C_{d,M} \vee \sqrt{\bar m(M,\mathbf{m})}M, \label{constdefDd}\\
 & \bar m(M,\mathbf{m}):=\min \left\{k \in \NN: m_k \geq M\right\}. \label{mkinvdef} 
 \end{align}
 Also, observe that $ \kl{P}{Q} \geq D_{\mathcal{G}_{k}^*(m_k)}(P,Q) $ since $g_{\theta_k^*} \in \mathcal{G}_{k}^*(m_k)$ is bounded. Then, the following chain of inequalities hold: 
\begin{flalign}
  &   \abs{\kl{P}{Q}-D_{\mathcal{G}_{k}^*(m_k)}(P,Q)} \notag \\
  &= \kl{P}{Q}- D_{\mathcal{G}_{k}^*(m_k)}(P,Q) \notag\\
      & \stackrel{(a)}{\leq} \mathbb{E}_P\left[\abs{f_{\mathsf{KL}}(X)- g_{\theta_k^*}(X)}\right]+L_{P,Q}~\mathbb{E}_Q\left[\abs{1-e^{g_{\theta_k^*}(Y)-f_{\mathsf{KL}}(Y)}}\right] \notag \\
    & \stackrel{(b)}{\leq}   D_{d,M,\mathbf{m}} k^{-\frac 12}+ e^M \left(e^{D_{d,M,\mathbf{m}} k^{-\frac 12}}-1\right), \label{approxbndklfstrcase2} && 
\end{flalign}
where 
\begin{enumerate}[label = (\alph*),leftmargin=*]
    \item follows similar to \eqref{intrconsteps};
    \item is due to \eqref{newconstdimdef} and $L_{P,Q} \leq e^M$ since $f_{\mathsf{KL}} \in \mathcal{I}(M)$.
\end{enumerate}
On the other hand,  taking $a_{1,k}=\sqrt{k} \log k$, $ka_{2,k}=a_{3,k}=m_k$, and  $k$ satisfying $\sqrt{k}e^{2m_k} =O \left(n^{\frac{1-\alpha}{2}}\right)$ for some $\alpha>0$,
 we have  
 \begin{flalign}
&  \mathbb{E}\left[  \abs{\hat{D}_{\mathcal{G}_{k}^*(m_k)}(X^n,Y^n) -\kl{P}{Q}}\right]  \notag \\
&\stackrel{(a)}{\leq} \abs{D_{\mathcal{G}_{k}^*(m_k)}(P,Q) -\kl{P}{Q}}   + \mathbb{E}\left[\abs{D_{\mathcal{G}_{k}^*(m_k)}(P,Q)-
    \hat{D}_{\mathcal{G}_{k}^*(M)}(X^n,Y^n)}\right] \notag\\
  &\stackrel{(b)}{\leq}  
  D_{d,M,\mathbf{m}} k^{-\frac 12}+ e^M \left(e^{D_{d,M,\mathbf{m}} k^{-\frac 12}}-1\right)+O\left(e^{2m_k}\sqrt{k}~ n^{-\frac 12}\right) \label{finerrbndklapp}\\
&\stackrel{(c)}{=}O_{M}\mspace{-5 mu}\left(e^{D_{d,M,\mathbf{m}} k^{-\frac 12}}-1\right)\mspace{-3 mu}+O\mspace{-3 mu}\left(\mspace{-3 mu}e^{2m_k}\sqrt{k} n^{-\frac 12}\mspace{-3 mu}\right), \label{KLdivgbndorder} &&
 \end{flalign}
 where 
 \begin{enumerate}[label = (\alph*),leftmargin=*]
    \item  is due to triangle inequality;
    \item follows from \eqref{kldivgempesterr} and \eqref{approxbndklfstrcase2}.
\end{enumerate}
Choosing $m_k=0.5 \log k$ in \eqref{KLdivgbndorder} yields
\begin{flalign}
 &  \mathbb{E}\left[  \abs{\hat{D}_{\mathcal{G}_{k}^*(0.5 \log k)}(X^n,Y^n)  -\kl{P}{Q}}\right]= O\mspace{-5 mu}\left(\mspace{-3 mu}k^{-\frac 12}\right)\mspace{-3 mu}+O\mspace{-3 mu}\left(\mspace{-3 mu}k^{\frac 32} n^{-\frac 12}\mspace{-3 mu}\right), \label{bnderrsimpkl}&&
\end{flalign} 
since  for $k$ sufficiently large,
    \begin{align}
      e^{D_{d,M,\mathbf{m}} k^{-\frac 12}}-1 & =\sum_{j=1}^{\infty}\frac{\left(D_{d,M,\mathbf{m}} k^{-\frac 12}\right)^j}{j !} \leq  \sum_{j=1}^{\infty}\left(D_{d,M,\mathbf{m}} k^{-\frac 12}\right)^j =O \left( k^{-\frac 12}\right). \notag
    \end{align}
  This completes the proof.
      \begin{remark}\label{remarkeffratekl}
 Setting $m_k=M$ in \eqref{KLdivgbndorder} and via steps leading to \eqref{bnderrsimpkl}, we obtain \eqref{finbnderrrate-case1}.
    \end{remark}
 \subsubsection{Proof of Lemma \ref{lem:consicomp}}\label{lem:consicomp-proof}
 Note that for $\gamma_{\mathsf{KL}}(x)=e^x-1$, 
 \begin{align}
    & \bar{\gamma}'_{\mathcal{G}_{k}(\mathbf{a}_k)}=\sup_{\substack{x \in \X, \\g_{\theta} \in \mathcal{G}_k(\mathbf{a}_k)} }\gamma_{\mathsf{KL}}'(g_{\theta}(x)) \leq e^{ka_{2,k}+a_{3,k}}, \notag \\
   &  R_{k,\mathbf{a}_k,\gamma} \leq 2\sqrt{k}\left(e^{ka_{2,k}+a_{3,k}}+1\right), \notag
 \end{align}
 where $\gamma_{\mathsf{KL}}'$ denotes the derivative of $\gamma_{\mathsf{KL}}$.  Since
 \begin{align}
  E_{k,\mathbf{a}_k,n,\gamma} \leq 4 \sqrt{2} n^{-\frac 12}k^{\frac 32} a_{2,k}\left(e^{ka_{2,k}+a_{3,k}}+1\right) \xrightarrow[n\rightarrow \infty]{} 0,   
 \end{align}
 for  $k,\mathbf{a}_k$ such that  $k^{\frac 32}a_{2,k}e^{ka_{2,k}+a_{3,k}} = O \left(n^{\frac{1-\alpha}{2}}\right)$ for $\alpha>0$, it follows from \eqref{bndesterremp} that for any $k \in \mathbb{N}$, $\delta>0$, and $n$ sufficiently large, 
 \begin{align}
     \mathbb{P}\left(\abs{D_{\mathcal{G}_k(\mathbf{a}_k)}(P,Q) -
     \hat{D}_{\mathcal{G}_{k}\left(\mathbf{a}_k\right)}(X^n,Y^n) } \geq \delta \right) \leq  2C e^{-\frac{n(\delta-CE_{k,\mathbf{a}_k,n,\gamma})^2}{16Ca_{2,k}^2k^2\left(e^{ka_{2,k}+a_{3,k}}+1\right)^2}}. \label{probdevbndkl}
 \end{align}
 Hence,  for  $k_n,\mathbf{a}_{k_n}$  such that   $k_n^{\frac 32}a_{2,k_n}e^{k_na_{2,k_n}+a_{1,k_n}} =O \left(n^{\frac{1-\alpha}{2}}\right)$,
\begin{flalign}
 & \sum_{n=1}^{\infty}  \mathbb{P}\left(\abs{D_{\mathcal{G}_{k_n}\left(\mathbf{a}_{k_n}\right)}(P,Q) -
     \hat{D}_{\mathcal{G}_{k_n}\left(\mathbf{a}_{k_n}\right)}(X^n,Y^n)} \geq \delta \right) \leq  2C\sum_{n=1}^{\infty}e^{-\frac{n(\delta-CE_{k,\mathbf{a}_k,n,\gamma})^2}{16Ca_{2,k_n}^2 k_n^2\left(e^{k_na_{2,k_n}+a_{1,k_n}}+1\right)^2}} <\infty, \label{finiteasbc} &&
\end{flalign}
where the final inequality in  \eqref{finiteasbc} can be established via integral test for sum of series. This implies \eqref{errbndest} via the first Borel-Cantelli lemma.
 To prove \eqref{kldivgempesterr}, note that
 \begin{flalign}
   &  \mathbb{E}\left[\abs{D_{\mathcal{G}_k(\mathbf{a}_k)}(P,Q) -
     \hat{D}_{\mathcal{G}_{k}\left(\mathbf{a}_k\right)}(X^n,Y^n)}\right] \notag \\
    &=\int_{0}^{\infty} \mathbb{P}\left(\abs{D_{\mathcal{G}_k(\mathbf{a}_k)}(P,Q) -
     \hat{D}_{\mathcal{G}_{k}\left(\mathbf{a}_k\right)}(X^n,Y^n)} \geq \delta \right)d\delta \notag \\
    & \leq CE_{k,\mathbf{a}_k,n,\gamma}+\int_{CE_{k,\mathbf{a}_k,n,\gamma}}^{\infty} 2C e^{-\frac{n(\delta-CE_{k,\mathbf{a}_k,n,\gamma})^2}{16Ca_{2,k}^2k^2\left(e^{ka_{2,k}+a_{3,k}}+1\right)^2}} d\delta \notag \\ &=O\left(n^{-\frac 12}k^{\frac 32}a_{2,k} e^{ka_{2,k}+a_{3,k}}\right). \label{kldiverrbndkn}  &&
 \end{flalign}
 \subsection{Proof of Proposition \ref{prop:distconddirect}}\label{prop:distconddirect-proof} 
From proof of Corollary \ref{cor:bndfourcoeff} (see \eqref{finextchar}), there exists extensions $f_{p}^{(\mathsf{e})},f_{q}^{(\mathsf{e})} \in \cB_{ b'   \vee c'} \cap \cS_{s,b'}\left(\RR^d\right)$ of $f,\bar f$, respectively (see \eqref{fourcoeffholdval1} and \eqref{fourcoeffholdval2} for definitions of $b'$ and $c'$). Define $f_{\mathsf{KL}}^{(\mathsf{e})}:=f_{p}^{(\mathsf{e})}-f_{q}^{(\mathsf{e})}$.  Note that since $f_{p}^{(\mathsf{e})},f_{q}^{(\mathsf{e})} \in \mathcal{L}^1\left(\RR^d\right) \cup \mathcal{L}^2\left(\RR^d\right)$, their Fourier transforms exists.  Hence,  we have  
\begin{align}
 & B \left(f_{\mathsf{KL}}^{(\mathsf{e})}\right) \stackrel{(a)}{\leq} B \left(f_{p}^{(\mathsf{e})}\right)+ B \left(f_{q}^{(\mathsf{e})}\right) \stackrel{(b)}{\leq} 2 ( b'   \vee c'),\label{extconst1}\\
  &\max_{x \in \X} \abs{f_{\mathsf{KL}}^{(\mathsf{e})}(x)} \leq \max_{x \in \X} \abs{f_{p}^{(\mathsf{e})}(x)}+\max_{x \in \X} \abs{f_{q}^{(\mathsf{e})}(x)} \stackrel{(d)}{\leq} 2b, \label{extconst3}
\end{align}
where
\begin{enumerate}[label = (\alph*),leftmargin=*]
    \item follows from the definition in \eqref{Cfconstdefval} and linearity of the Fourier transform;
    \item (c) is since $f_{p}^{(\mathsf{e})},f_{q}^{(\mathsf{e})} \in \cB_{ b'   \vee c'} $; \stepcounter{enumi}
     \item is due to $(P,Q) \in   \mathcal{L}_{\mathsf{KL}}(b,c)$.
\end{enumerate}
Hence, it follows from \eqref{extconst1}-\eqref{extconst3} that $f_{\mathsf{KL}}^{(\mathsf{e})}|_\cX \in \mathcal{I}(M) $ with $M= 2 \bar c_{b,c,d}$ (since $b \leq b'$), where $\bar c_{b,c,d}$ is given in \eqref{constapproxhold}. The claim then follows from Theorem \ref{strongcons} since $f_{\mathsf{KL}}=f_{\mathsf{KL}}^{(\mathsf{e})}|_\cX$. 
 \section{Appendix: $\chi^2$ divergence}
\subsection{Proof of Theorem \ref{strongconschisq}} \label{strongconschisq-proof}
Let $\chi^2_{\mathcal{G}_k(\mathbf{a}_k)}(P,Q):=\mathsf{H}_{\gamma_{\chi^2}, \mathcal{G}_k(\mathbf{a}_k)}(P,Q)$. 
The proof of Theorem  \ref{strongconschisq} is based on the lemma below (see  Appendix \ref{lem:consicompchisq-proof} for proof).
 \begin{lemma} \label{lem:consicompchisq}
Let  $P,Q\in\mathcal{P}_{\chi^2}(\X)$. For $X^n \sim P^{\otimes n}$ and $Y^n \sim Q^{\otimes n}$, the following holds for any $\alpha>0$:
\begin{enumerate}[label = (\roman*),leftmargin=*]
    \item For  $n,k_n,\mathbf{a}_{k_n}$ such that $k_n^{\frac 52}a_{2,k_n}^2+k_n^{\frac 32}a_{2,k_n} a_{3,k_n}= O \left(n^{\frac{1-\alpha}{2}}\right)$,
\begin{align}
  \hat{\chi^2}_{\mathcal{G}_{k}\left(\mathbf{a}_{k_n}\right)}(X^n,Y^n)   \xrightarrow[n\rightarrow \infty]{}     \chi^2_{\mathcal{G}_{k_n}\left(\mathbf{a}_{k_n}\right)}(P,Q),\quad  \mathbb{P}-\mbox{a.s.} \label{errbndestchisq}
\end{align}
\item For  $n,k,\mathbf{a}_k$ such that $k^{\frac 52}a_{2,k}^2+k^{\frac 32}a_{2,k} a_{3,k}= O \left(n^{\frac{1-\alpha}{2}}\right)$, \begin{flalign}
     \mathbb{E}\left[\abs{\hat{\chi^2}_{\mathcal{G}_{k}\left(\mathbf{a}_k\right)}(X^n,Y^n)-
    \chi^2_{\mathcal{G}_k(\mathbf{a}_k)}(P,Q)}\right] =O\left(n^{-\frac 12}\left(k^{\frac 52}a_{2,k}^2+k^{\frac 32}a_{2,k}a_{3,k} \right)\right). \label{chisqdivgempesterr} &&
 \end{flalign}
\end{enumerate}
\end{lemma}
The proof of \eqref{finbndasconchisq} follows from  \eqref{errbndestchisq}, using similar arguments used to establish \eqref{finbndascon} and steps leading to \eqref{chisqfinbndrate} below. The details are omitted.

  We proceed to prove \eqref{chisqsimprate}.  
Since $f_{\chi^2} \in \mathcal{I}(M)$, we have similar to \eqref{newconstdimdef} that  there exists $g_{\theta_k^*} \in \mathcal{G}_{k}^*(m_k)$ 
\begin{align}
     &  \norm{ f_{\chi^2}-  g_{\theta_k^*}}_{\infty,P,Q} =D_{d,M,\mathbf{m}}k^{-\frac 12}, \label{sqrtbndappchisq}
\end{align}
where $D_{d,M,\mathbf{m}}$ is defined in \eqref{constdefDd}.
Also, $ \chisq{P}{Q} \geq \chi^2_{\mathcal{G}_{k}^*(m_k)}(P,Q) $ since $g_{\theta} \in \mathcal{G}_{k}^*(m_k)$ is bounded. 
  Then, we have
\begin{flalign}
    & \abs{\chisq{P}{Q}- \chi^2_{\mathcal{G}_{k}^*(m_k)}(P,Q)} \notag \\
    &= \chisq{P}{Q}- \chi^2_{\mathcal{G}_{k}^*(m_k)}(P,Q) \notag\\
      &\leq \chisq{P}{Q}-  \mathbb{E}_P[g_{\theta_k^*}(X)]-\mathbb{E}_Q\left[g_{\theta_k^*}(Y)+\frac{g_{\theta_k^*}^2(Y)}{4}\right] \notag\\
     &\leq \mathbb{E}_P\left[\abs{f_{\chi^2}(X)- g_{\theta_k^*}(X)}\right]+\mathbb{E}_Q\left[\abs{f_{\chi^2}(Y)- g_{\theta_k^*}(Y)}+\frac 14\abs{f_{\chi^2}^2(Y)- g_{\theta_k^*}^2(Y)}\right] \notag\\
     &\leq  2D_{d,M,\mathbf{m}}k^{-\frac 12}+\mathbb{E}_Q\left[\frac 14\abs{f_{\chi^2}(Y)- g_{\theta_k^*}(Y)}\abs{f_{\chi^2}(Y)+ g_{\theta_k^*}(Y)}\right]\notag\\
     &\leq  2 D_{d,M,\mathbf{m}}  k^{-\frac 12}+\mathbb{E}_Q\Big[\frac 14\abs{f_{\chi^2}(Y)- g_{\theta_k^*}(Y)}\abs{ g_{\theta_k^*}(Y)-f_{\chi^2}(Y)}  +\frac 12\abs{f_{\chi^2}(Y)- g_{\theta_k^*}(Y)}\abs{f_{\chi^2}(Y)}\Big] \notag \\
     & \leq  2 D_{d,M,\mathbf{m}}  k^{-\frac 12}+ \frac{D_{d,M,\mathbf{m}}^2}{4k} +\frac{D_{d,M,\mathbf{m}} M}{2\sqrt{k}}, \label{chisqfinbndrate}&&
\end{flalign}
where \eqref{chisqfinbndrate} is due to $f_{\chi^2} \in \mathcal{I}(M)$. Taking $a_{1,k}=\sqrt{k} \log k$, $ka_{2,k}=a_{3,k}=m_k$, and  $k,m_k$ satisfying  $m_k^2 \sqrt{k} =O\left(n^{(1-\alpha)/2}\right)$, we have
 \begin{flalign}
&     \mathbb{E}\left[  \abs{\hat{\chi^2}_{\mathcal{G}_{k}^*(m_k)}(X^n,Y^n) -\chisq{P}{Q}}\right] \notag \\
&\stackrel{(a)}{\leq} \abs{ \chi^2_{\mathcal{G}_{k}^*(m_k)}(P,Q) -\chisq{P}{Q}} +    \mathbb{E}\left[\abs{\chi^2_{\mathcal{G}_{k}^*(m_k)}(P,Q)-
    \hat{\chi^2}_{\mathcal{G}_{k}^*(m_k)}(X^n,Y^n)}\right] \notag \\
   & \stackrel{(b)}{\leq}  2 D_{d,M,\mathbf{m}} k^{-\frac 12}+ \frac{D_{d,M,\mathbf{m}}^2}{4k} +\frac{D_{d,M,\mathbf{m}}M}{2\sqrt{k}} +O\left(m_k^2 \sqrt{k}~ n^{-\frac 12}  \right), \label{finerrbndchisqapp} \\
  &\stackrel{(c)}{=}O_{d,M}\left(\bar m(M,\mathbf{m})k^{-\frac{1}{2}}\right)+O\left(m_k^2 \sqrt{k}~ n^{-\frac 12}  \right), \notag &&
 \end{flalign}
 where
 \begin{enumerate}[label = (\alph*),leftmargin=*]
    \item is due to triangle inequality;
    \item follows from  \eqref{chisqdivgempesterr} and \eqref{chisqfinbndrate}; 
     \item is by the definition of $D_{d,M,\mathbf{m}}$ in \eqref{constdefDd} and since $\bar m(M,\mathbf{m}) \geq 1$.
\end{enumerate}
Setting $\mathbf{m}=\{0.5 \log k\}_{k \in \NN}$ in \eqref{finerrbndchisqapp} yields \eqref{chisqsimprate},   thus completing the proof.
 
 \subsubsection{Proof of Lemma \ref{lem:consicompchisq}} \label{lem:consicompchisq-proof}
For $\gamma_{\chi^2}(x)=x+\frac{x^2}{4}$, we have
\begin{align}
   &  \bar{\gamma}'_{\mathcal{G}_{k}(\mathbf{a}_k)}=\sup_{\substack{x \in \X, \\g_{\theta} \in \mathcal{G}_k(\mathbf{a}_k)} }\gamma_{\chi^2}'(g_{\theta}(x)) \leq 0.5(ka_{2,k}+a_{3,k})+1, \notag \\
   & R_{k,\mathbf{a}_k,\gamma} \leq 2\sqrt{k}\left(0.5(ka_{2,k}+a_{3,k})+2\right),
\end{align}
where $\gamma_{\chi^2}'(\cdot)$ denotes the derivative of $\gamma_{\chi^2}$. Since
\begin{align}
  0 \leq   E_{k,\mathbf{a}_k,n,\gamma} \leq 4 \sqrt{2} n^{-\frac 12}k^{\frac 32} a_{2,k}\left(0.5(ka_{2,k}+a_{3,k})+2\right) \xrightarrow[n\rightarrow \infty]{}  0,
\end{align}
 for  $k,\mathbf{a}_k$ such that  $k^{\frac 52}a_{2,k}^2+k^{\frac 32}a_{2,k} a_{3,k}= O \left(n^{\frac{1-\alpha}{2}}\right)$, it follows from \eqref{bndesterremp} that for any $k \in \mathbb{N}$, $\delta>0$, and $n$ sufficiently large, 
 \begin{align}
     \mathbb{P}\left(\abs{\hat{\chi^2}_{\mathcal{G}_{k}\left(\mathbf{a}_k\right)}(X^n,Y^n)-\chi^2_{\mathcal{G}_{k}\left(\mathbf{a}_k\right)}(P,Q)} \geq \delta \right) \leq  2C e^{-\frac{n(\delta-CE_{k,\mathbf{a}_k,n,\gamma})^2}{16Ca_{2,k}^2k^2\left(0.5(ka_{2,k}+a_{3,k})+2\right)^2}}.
 \end{align}
 Then, \eqref{errbndestchisq} and  \eqref{chisqdivgempesterr} follows using similar steps used to prove \eqref{errbndest} (see  \eqref{finiteasbc}) and \eqref{kldivgempesterr} (see \eqref{kldiverrbndkn}) in Theorem  \ref{strongcons}, respectively.  This completes the proof.
 \subsection{Proof of Proposition \ref{prop:distconddirect-chisq}} \label{prop:distconddirect-chisq-proof}
It follows from \eqref{finextchar} that there exists  extensions $f_{p}^{(\mathsf{e})},f_{q}^{(\mathsf{e})} \in \cB_{ b'   \vee c'} \cap \tilde{\cS}_{s,b'}\left(\RR^d\right)$ of $f, \bar f \in \cH_{b,c}^{s,\delta}(\Ucal)$, respectively, where $ \tilde{\cS}_{s,b'}\left(\RR^d\right)$ is defined in \eqref{squareintclassallder}. Let $f_{\chi^2}^{(\mathsf{e})}=2 \left(f_{p}^{(\mathsf{e})} \cdot f_{q}^{(\mathsf{e})}-1\right) $. Recall the notation   $\boldsymbol{\alpha}_j$ for a multi-index of order $j$. We have from the chain rule for differentiation that $D^{\boldsymbol{\alpha}_j}f_{\chi^2}^{(\mathsf{e})}(x) $ is the sum of $2^j$ terms of the form $D^{\boldsymbol{\alpha}_{j_1}}f_{p}^{(\mathsf{e})}(x)\cdot D^{\boldsymbol{\alpha}_{j_2}}f_{q}^{(\mathsf{e})}(x)$, where $\boldsymbol{\alpha}_{j_1}+\boldsymbol{\alpha}_{j_2}=\boldsymbol{\alpha}_{j}$. Also, note  from \eqref{derbndinu} and \eqref{allderivsqint} that for $j=0,\ldots,s$, $f_{p}^{(\mathsf{e})}$,  $f_{q}^{(\mathsf{e})}$ satisfies
\begin{subequations}
\begin{equation}
   \abs{D^{\boldsymbol{\alpha_j}}f_{p}^{(\mathsf{e})}(x)} \vee \abs{D^{\boldsymbol{\alpha_j}}f_{q}^{(\mathsf{e})}(x)} \leq \hat b \leq b',~\forall~ x \in \RR^d,
\end{equation}
\begin{equation}
 \norm{D^{\boldsymbol{\alpha}_{j}}f_{p}^{(\mathsf{e})}}_{L_2\left(\RR^d\right)} \vee \norm{D^{\boldsymbol{\alpha}_{j}}f_{q}^{(\mathsf{e})}}_{L_2\left(\RR^d\right)}  \leq b'. 
\end{equation}
\end{subequations}
Then, it follows that for $j=0,\ldots,s$,
\begin{align}
  \norm{D^{\boldsymbol{\alpha}_j}f_{\chi^2}^{(\mathsf{e})}}_{L_2\left(\RR^d\right)}& \leq 2+2\norm{\sum_{\substack{\boldsymbol{\alpha}_{j_1},\boldsymbol{\alpha}_{j_2}:\\\boldsymbol{\alpha}_{j_1}+\boldsymbol{\alpha}_{j_2}=\boldsymbol{\alpha}_j}}D^{\boldsymbol{\alpha}_{j_1}}f_{p}^{(\mathsf{e})} \cdot D^{\boldsymbol{\alpha}_{j_2}}f_{q}^{(\mathsf{e})}}_{L_2\left(\RR^d\right)} \notag \\
  &\leq 2+ 2^{j+1} b'\max_{\boldsymbol{\alpha}_{j_2}}\norm{D^{\boldsymbol{\alpha}_{j_2}}f_{q}^{(\mathsf{e})}}_{L_2\left(\RR^d\right)} \notag \\
  &\leq 2+2^{j+1} b'^2. \label{chisqconstbndprop}
\end{align}
Hence, $f_{\chi^2}^{(\mathsf{e})} \in \tilde{\cS}_{s,2+2^{s+1} b'^2}\left(\RR^d\right)$. From Lemma  \ref{propsuffcond}, it follows that $B \left(f_{\chi^2}^{(\mathsf{e})}\right) \leq (2+2^{s+1} b'^2) \kappa_d \sqrt{d}$.   Moreover, we have 
\begin{align}
    \sup_{x \in \X} \abs{f_{\chi^2}^{(\mathsf{e})}} \leq  2+2 \sup_{x \in \X} \frac{p(x)}{q(x)} \leq  2+2b^2.
\end{align}
This implies that $f_{\chi^2}^{(\mathsf{e})}|_\cX \in \mathcal{I}\left((2+2^{s+1}b'^2)~ (\kappa_d \sqrt{d}\vee 1) \right)$ since $b' \geq b$.
The claim then follows from Theorem \ref{strongconschisq} by noting that $f_{\chi^2}=f_{\chi^2}^{(\mathsf{e})}|_\cX$ and $b'^2 \leq \bar c_{b,c,d}^2$.

 \section{Appendix: Squared Hellinger distance}
\subsection{Proof of Theorem \ref{strongconshel}} \label{strongconshel-proof}
Let $H^2_{\tilde {\mathcal{G}}_k(\mathbf{a}_k,t)}(P,Q):=\mathsf{H}_{\gamma_{H^2}, \tilde {\mathcal{G}}_k(\mathbf{a}_k,t)}(P,Q)$. 
The proof of Theorem  \ref{strongconshel} hinges on the following lemma, whose proof is given in Appendix \ref{lem:consicomphel-proof}.
\begin{lemma} \label{lem:consicomphel}
Let  $P,Q\in\mathcal{P}_{H^2}(\X)$. For $X^n \sim P^{\otimes n}$ and $Y^n \sim Q^{\otimes n}$, the following holds for any $\alpha>0$:
\begin{enumerate}[label = (\roman*),leftmargin=*]
    \item For  $n,k_n,\mathbf{a}_{k_n}$ such that $k_n^{\frac 32}a_{2,k_n} t_{k_n}^{-2}= O \left(n^{\frac{1-\alpha}{2}}\right)$,
\begin{align}
  \hat{H^2}_{\tilde{\mathcal{G}}_{k_n}\left(\mathbf{a}_{k_n},t_{k_n}\right)}(X^n,Y^n)   \xrightarrow[n\rightarrow \infty]{}     H^2_{\tilde{\mathcal{G}}_{k_n}\left(\mathbf{a}_{k_n},t_{k_n}\right)}(P,Q),\quad  \mathbb{P}-\mbox{a.s.} \label{errbndesthel}
\end{align}
\item For  $n,k,\mathbf{a}_k$ such that $k^{\frac 32}a_{2,k} t_k^{-2}= O \left(n^{\frac{1-\alpha}{2}}\right)$, \begin{flalign}
     \mathbb{E}\left[\abs{\hat{H^2}_{\tilde{\mathcal{G}}_{k}\left(\mathbf{a}_k,t_k\right)}(X^n,Y^n)-
    H^2_{\tilde{\mathcal{G}}_{k}\left(\mathbf{a}_k,t_k\right)}(P,Q)}\right] =O\left(n^{-\frac 12}k^{\frac 32}a_{2,k} t_k^{-2}\right). \label{heldivgempesterr} &&
 \end{flalign} 
\end{enumerate}
\end{lemma}
We first prove \eqref{finbndasconhel}. Since  $f_{H^2} \in \mathsf{C}\left(\X\right)$ for a compact set $\X$, its supremum is achieved at some $x^* \in \X$. Also, since $\norm{\frac{dP}{dQ}}_{\infty}<\infty$ by definition of the Radon-Nikodym derivative,  we have $\sup_{x \in \X}f_{H^2}(x)=f_{H^2}(x^*)<1$. Moreover,  $t_k \leq 1-f_{H^2}(x^*)$ for sufficiently large $k$ since $t_k \rightarrow 0$. Then, it follows from \citet[Theorem 2.8]{stinchcombe1990approximating} that for any $\epsilon>0$ and $k \geq k_0(\epsilon)$ (some integer), there exists a $g_{\theta^*} \in \tilde {\mathcal{G}}^{(1)}_{k,t_k}$ such that
  \begin{align}
      \sup_{x \in \X}\abs{f_{H^2}(x)-g_{\theta^*}(x)} \leq \epsilon. \label{approxbndwthel}
  \end{align}
 This implies similar to \eqref{approxlim} in Theorem \ref{strongcons} that
  \begin{align}
     \lim_{k \rightarrow \infty}  H^2_{\tilde {\mathcal{G}}^{(1)}_{k,t_k}}(P,Q) =H^2(P,Q). \label{approxlimhel}
  \end{align} 
  Then, \eqref{finbndasconhel}  follows from \eqref{errbndesthel} and \eqref{approxlimhel}. 
  
Next, we prove \eqref{helsimperrbnd}.  Since $f_{H^2} \in \mathcal{I}_{H^2}(M)$, $1-f_{H^2}(x) \geq \frac{1}{M}$ for all $x \in \X$. Using $t_k \rightarrow 0$, we have from \eqref{approxratefin} that for $k$ such that $t_k  \leq \frac{1}{M}$ and $m_k \geq M$, there exists $g_{\theta} \in \tilde{\mathcal{G}}_{k,m_k,t_k}^{(2)}$ such that 
\begin{align}
     &  \norm{ f_{H^2}-  g_{\theta}}_{\infty,P,Q} \leq \tilde C_{d,M}k^{-\frac 12 }. \label{sqrtbndhel}
\end{align}
On the other hand, for $k$ such that $t_k  > \frac{1}{M}$ or $m_k <M$,  taking $g_{\mathbf{0}}=0$ yields $ \norm{ f_{H^2}-  g_{\mathbf{0}}}_{\infty,P,Q} \leq M$ as $f_{H^2}\in \mathcal{I}(M)$. Then, denoting $\mathbf{t}=\{t_k\}_{k \in \NN}$,  it follows similar to \eqref{newconstdimdef} that for all $k$, there exists $g_{\theta_k^*} \in \tilde{\mathcal{G}}_{k,m_k,t_k}^{(2)}$ such that 
 \begin{align}
     \norm{ f_{H^2}-  g_{\theta_k^*}}_{\infty,P,Q} \leq \tilde C_{d,M} k^{-\frac 12 } \vee\left( \sqrt{\bar t \left(M^{-1},\mathbf{t}\right)} \vee \sqrt{\bar m(M,\mathbf{m})}\right)Mk^{-\frac 12 }=:\bar D_{d,M,\mathbf{t},\mathbf{m}}k^{-\frac 12},\label{newconstdimdefhel}
 \end{align}
where $\bar t \left(M^{-1},\mathbf{t}\right):=\inf\{k:t_k \leq M^{-1}\}$.
 Moreover, note that by definition, $ H^2(P,Q) \geq   H^2_{\tilde {\mathcal{G}}^{(2)}_{k,m_k,t_k}}(P,Q)$. 
 Then, we have
\begin{flalign}
    &\abs{H^2(P,Q)- H^2_{\tilde {\mathcal{G}}^{(2)}_{k,m_k,t_k}}(P,Q)}  \notag \\
    &= H^2(P,Q)- H^2_{\tilde {\mathcal{G}}^{(2)}_{k,m_k,t_k}}(P,Q)  \notag\\
 &\leq \mathbb{E}_P\left[f_{H^2}(X)\right]-\mathbb{E}_Q\left[\frac{f_{H^2}(Y)}{1-f_{H^2}(Y)}\right]- \mathbb{E}_P\left[g_{\theta_k^*}(X)\right]+\mathbb{E}_Q\left[\frac{g_{\theta_k^*}(Y)}{1-g_{\theta_k^*}(Y)}\right] \notag \\
     &\leq \mathbb{E}_P\left[\abs{f_{H^2}(X)- g_{\theta_k^*}(X)}\right]+\mathbb{E}_Q\left[\abs{\frac{f_{H^2}(Y)}{1-f_{H^2}(Y)}-\frac{g_{\theta_k^*}(Y)}{1-g_{\theta_k^*}(Y)}}\right] \notag\\
     &\leq \bar D_{d,M,\mathbf{t},\mathbf{m}}k^{-\frac 12}+\mathbb{E}_Q\left[ \abs{\frac{f_{H^2}(Y)- g_{\theta_k^*}(Y)}{(1-f_{H^2}(Y))(1-g_{\theta_k^*}(Y))}}\right]\notag\\
      & \leq  \bar D_{d,M,\mathbf{t},\mathbf{m}}k^{-\frac 12}+ M~t_k^{-1} \bar D_{d,M,\mathbf{t},\mathbf{m}}k^{-\frac 12},
     \label{approxerrhelcase2}&&
\end{flalign}
where \eqref{approxerrhelcase2} is  due to $1-g_{\theta^*}(x) \geq t_k$, $\left(1-f_{H^2}(x)\right)^{-1} \leq M$ for all $x \in \X$, and \eqref{newconstdimdefhel}. 

Then, it follows  from \eqref{heldivgempesterr} and \eqref{approxerrhelcase2} that by taking $a_{1,k}=\sqrt{k} \log k$, $ka_{2,k}=a_{3,k}=m_k$, and  $\sqrt{k}m_kt_k^{-2} =O \left(n^{(1-\alpha)/2}\right)$ for some $\alpha>0$, we have
 \begin{flalign}
&  \mathbb{E}\left[  \abs{\hat{H}^2_{\tilde{\mathcal{G}}_{k,m_k,t_k}^{(2)}}(X^n,Y^n) -H^2(P,Q)}\right] \notag \\
&\leq \abs{H^2(P,Q)- H^2_{\tilde {\mathcal{G}}^{(2)}_{k,m_k,t_k}}(P,Q)}   + \mathbb{E}\left[\abs{\hat{H^2}_{\tilde{\mathcal{G}}_{k,m_k,t_k}^{(2)}}(X^n,Y^n)-
    H^2_{\tilde{\mathcal{G}}_{k,m_k,t_k}^{(2)}}(P,Q)}\right] \notag \\
 &\leq   \bar D_{d,M,\mathbf{t},\mathbf{m}}k^{-\frac 12}+ M~t_k^{-1} \bar D_{d,M,\mathbf{t},\mathbf{m}}k^{-\frac 12} +O\mspace{-4mu}\left(\mspace{-2.5mu}m_k \sqrt{k}t_k^{-2} n^{\mspace{-2mu}-\frac 12}\mspace{-2mu}\right) \label{finbndratehelcase2}\\
 &=O_{d,M}\mspace{-4mu}\left(\mspace{-2mu}\sqrt{\bar t \left(M^{-1},\mathbf{t}\right)} \vee \sqrt{\bar m(M,\mathbf{m})}~t_k^{-1}k^{-\frac{1}{2}}\mspace{-2.5mu}\right) +O\mspace{-4mu}\left(\mspace{-2.5mu}m_k \sqrt{k}t_k^{-2} n^{\mspace{-2mu}-\frac 12}\mspace{-2mu}\right). \notag && 
 \end{flalign}
 Setting $m_k=0.5 \log k$ and $t_k=\log^{-1} k$ in \eqref{finbndratehelcase2} yields \eqref{helsimperrbnd}, thus completing the proof.
 \subsubsection{Proof of Lemma \ref{lem:consicomphel}} \label{lem:consicomphel-proof}
 Note that   Theorem \ref{empesterrbnd} continues to hold with $ \mathcal{G}_k(\mathbf{a})$ in \eqref{maxdergamma} and \eqref{bndesterremp} replaced with $\tilde{ \mathcal{G}}_k(\mathbf{a},t)$, since for $\gamma_{H^2}(x)=\frac{x}{1-x}$, 
\begin{align}
   &  \bar{\gamma}'_{\tilde{ \mathcal{G}}_k(\mathbf{a}_k,t_k)}=\sup_{\substack{x \in \X, \\g_{\theta} \in \tilde{ \mathcal{G}}_k(\mathbf{a}_k,t_k)} }\gamma_{H^2}'(g_{\theta}(x)) =\sup_{\substack{x \in \X, \\g_{\theta} \in \tilde{ \mathcal{G}}_k(\mathbf{a}_k,t_k)} }\frac{1}{(1-g_{\theta})^2} \leq \frac{1}{t_k^2}, \notag 
\end{align}
where $\gamma_{H^2}'(\cdot)$ denotes the derivative of $\gamma_{H^2}$. This implies that $ R_{k,\mathbf{a}_k,\gamma} \leq 2\sqrt{k}\left(t_k^{-2}+1\right)$, and 
\begin{align}
  0 \leq   E_{k,\mathbf{a}_k,n,\gamma} \leq 4 \sqrt{2} n^{-\frac 12}k^{\frac 32} a_{2,k}\left(t_k^{-2}+1\right) \xrightarrow[n\rightarrow \infty]{}  0, \notag
\end{align}
 for  $k,\mathbf{a}_k$, $t_k$ such that  $k^{\frac 32}a_{2,k} t_k^{-2}= O \left(n^{\frac{1-\alpha}{2}}\right)$. It then follows from \eqref{bndesterremp} that for any $k \in \mathbb{N}$, $\delta>0$, and $n$ sufficiently large, 
 \begin{align}
     \mathbb{P}\left(\abs{\hat{H^2}_{\mathcal{G}_{k}\left(\mathbf{a}_k\right)}(X^n,Y^n)-H^2_{\mathcal{G}_{k}\left(\mathbf{a}_k\right)}(P,Q)} \geq \delta \right) \leq  2C e^{-\frac{n(\delta-CE_{k,\mathbf{a}_k,n,\gamma})^2}{16Ca_{2,k}^2k^2\left(t_k^{-2}+1\right)^2}}. \notag
 \end{align}
 Then, \eqref{errbndesthel} and  \eqref{heldivgempesterr} follows using similar steps used to prove \eqref{errbndest} (see  \eqref{finiteasbc}) and \eqref{kldivgempesterr} (see \eqref{kldiverrbndkn}) in Theorem  \ref{strongcons}, respectively.  This completes the proof.


\end{document}